\documentclass[a4paper,12pt]{amsart}
\usepackage{standalone}
\usepackage{amsfonts}
\usepackage{amsthm}
\usepackage{amssymb}
\usepackage{amsmath}
\usepackage{theoremref}
\usepackage{enumerate}
\usepackage{bbm}
\usepackage{bm}
\usepackage{pgfplots}
\pgfplotsset{compat=1.18} 
\usepgfplotslibrary{fillbetween}
\usetikzlibrary{intersections}
\usepackage{mathtools}
\usetikzlibrary{patterns}

\numberwithin{equation}{section}

\newcommand{\A}{\mathcal{A}}

\renewcommand{\S}{\mathcal{S}}
\newcommand{\T}{\mathbb{T}}

\newcommand{\conj}[1]{\overline{#1}}
\newcommand{\D}{\mathbb{D}}

\newcommand{\R}{\mathbb{R}}

\newcommand{\Po}{\mathcal{P}}
\newcommand{\cD}{\conj{\mathbb{D}}}

\newcommand{\dist}[2]{\text{dist}( #1, #2 ) }

\newcommand{\m}{\textit{m}}
\newcommand{\hil}{\mathcal{H}}

\newcommand{\hb}{\mathcal{H}(b)}

\renewcommand\Re{\operatorname{Re}}

\newcommand{\coreh}[2]{\text{core}_{#1}(#2)}
\newcommand{\resh}[2]{\text{res}_{#1}(#2)}
\newcommand{\core}[1]{\text{core}(#1)}
\newcommand{\res}[1]{\text{res}(#1)}
\newcommand{\BCh}{\mathcal{BC}_h}
\newcommand{\BC}{\mathcal{BC}}
\newcommand{\assoch}[2]{\text{Assoc}\BC_{#1}(#2)}
\newcommand{\assoc}[1]{\text{Assoc}\BC(#1)}

\newtheorem{mainthm}{Theorem}

\newtheorem{thm}{Theorem}[section]
\newtheorem*{thm*}{Theorem}
\newtheorem{lem}[thm]{Lemma}

\newtheorem{cor}[thm]{Corollary}
\newtheorem*{cor*}{Corollary}
\newtheorem{prop}[thm]{Proposition}
\theoremstyle{definition}

\theoremstyle{definition}

\newtheorem{defn}[thm]{Definition}
\newtheorem*{defn*}{Definition}
\newtheorem{question}{Question}

\newtheorem{claim*}{Claim}

\let\oldtocsection=\tocsection

\let\oldtocsubsection=\tocsubsection

\let\oldtocsubsubsection=\tocsubsubsection

\renewcommand{\tocsection}[2]{\hspace{0em}\oldtocsection{#1}{#2}}
\renewcommand{\tocsubsection}[2]{\hspace{1em}\oldtocsubsection{#1}{#2}}
\renewcommand{\tocsubsubsection}[2]{\hspace{2em}\oldtocsubsubsection{#1}{#2}}

\def\section{\@startsection{section}{1}%
\z@{.7\linespacing\@plus\linespacing}{.5\linespacing}%
{\large\scshape\centering}}

\title[Weighted Korenblum-Roberts theory]{\large Weighted Korenblum-Roberts theory}
\author{Bartosz Malman}

\address{Division of Mathematics and Physics, 
        Mälardalen University,
		721 23 Västerås, Sweden}
\email{bartosz.malman@mdu.se} 

\begin{document}

\begin{abstract}
The classical Korenblum-Roberts Theorem characterizes the cyclic singular inner functions in the Bergman spaces of the unit disk $\mathbb{D}$ as those for which the corresponding singular measure vanishes on Beurling-Carleson sets of Lebesgue measure zero. We solve the weighted variant of the problem in which the Bergman space is replaced by a $\mathcal{P}^t(\mu)$ space, the closure of analytic polynomials in a Lebesgue space $\mathcal{L}^t(\mu)$ corresponding to a measure of the form $dA_\alpha + w\, d\m$, with $dA_\alpha$ being the standard weighted area measure on $\mathbb{D}$, $d\m$ the Lebesgue measure on the unit circle $\mathbb{T}$, and $w$ a general weight on $\mathbb{T}$. We characterize when $\mathcal{P}^t(\mu)$ of this form is a space of analytic functions on $\mathbb{D}$ by computing the Thomson decomposition of the measure $\mu$. The structure of the decomposition is expressed in terms of what we call the family of \textit{associated Beurling-Carleson sets}. We characterize the cyclic singular inner functions in the analytic $\mathcal{P}^t(\mu)$ spaces as those for which the corresponding singular measure vanishes on the family of associated Beurling-Carleson sets. Unlike the classical setting, Beurling-Carleson sets of both zero and positive Lebesgue measure appear in our description. As an application of our results, we complete the characterization of the symbols $b:\mathbb{D} \to \mathbb{D}$ which generate a de Branges-Rovnyak space with a dense subset of functions smooth on $\mathbb{T}$. The characterization is given explicitly in terms of the modulus of $b$ on $\T$ and the singular measure corresponding to the singular inner factor of $b$. Our proofs involve Khrushchev's techniques of simultaneous polynomial approximations and linear programming ideas of Korenblum, combined with recently established constrained $\mathcal{L}^1$-optimization tools.
\end{abstract}

\thanks{This research is supported by Vetenskapsrådet (VR2024-03959).}


\maketitle

\clearpage

\tableofcontents

\clearpage

\section{Introduction}

\subsection{A bit of background} Let $\Po$ denote the set of analytic polynomials in a single complex variable $z$, \[\Po := \Bigg\{ \sum_{i=0}^n p_iz^i : p_i \in \mathbb{C}, n \in \mathbb{N} \cup \{ 0 \} \Bigg\}\] and let $\mu$ be a finite non-negative Borel measure compactly supported in the complex plane $\mathbb{C}$. Given such a measure and a number $t \in (0,\infty)$, the closure of $\Po$ in the Lebesgue space $\mathcal{L}^t(\mu)$ is customarily denoted by $\Po^t(\mu)$. If $d\mu = d\m$, the Lebesgue measure on the unit circle $\T = \{ z \in \mathbb{C} : |z| = 1\}$, then $\Po^t(\mu)$ is the classical Hardy space $\hil^t$ of analytic functions on the unit disk $\D = \{ z \in \mathbb{C} : |z| < 1\}$, while if $d\mu = dA$ is the area measure on $\D$, then $\Po^t(\mu)$ is the Bergman space of area $t$-integrable functions on $\D$. In contrast, if $d\mu = dx_I$, the Lebesgue measure on an interval $I \subset \mathbb{R}$, then $\Po^t(\mu) = \mathcal{L}^t(\mu)$ is the Lebesgue space itself, a space of measurable functions exhibiting no analytic properties. In general, the nature of $\Po^t(\mu)$ is complicated, and the three mentioned examples illustrate in a sense the extreme cases. By the decomposition theorem of Thomson from \cite{thomson1991approximation}, the space $\Po^t(\mu)$ can always be decomposed into pieces, $\Po^t(\mu) = \oplus_i \Po^t(\mu_i)$, which with at most one exception are \textit{analytic}, in the sense that functions in the space are analytic on some open set. The potential exceptional piece is a full Lebesgue space.

The identification of the pieces $\mu_i$ for a general planar measure $\mu$ is a task too ambitious for the present work. In regards to this we mention that identification of the pieces in the Thomson decomposition of a deceptively simple-looking measure $\mu$ composed of two pieces, a radially weighted area measure on $\D$ with very fast decay near $\T$, and a weighted Lebesgue measure on $\T$, is essentially equivalent to Volberg's deep theorem from \cite{volberg1982logarithm} on summability of the logarithmic integral of functions with appropriate one-sided Fourier decay (see also the exposition by Volberg and Jöricke in \cite{vol1987summability}, and the introductory sections of \cite{kriete1990mean} and \cite{volberg1991criterion}). Other results of this type can be deduced from Khrushchev's profound work in \cite{khrushchev1978problem} on simultaneous polynomial approximation, which can be used to establish Thomson decompositions with applications to the theory of the Cauchy integral operator. In \cite{limani2023problem} it was found that the structure of invariant subspaces of the shift operator $M_z: f(z) \mapsto zf(z)$ acting on certain $\Po^2(\mu)$ spaces is connected with the problem of approximation by functions of various boundary regularity in de Branges-Rovnyak spaces $\hb$. In this context, an issue arises which calls for an extension of the results of Korenblum from \cite{korenblum1981cyclic} and Roberts from \cite{roberts1985cyclic} on cyclic singular inner functions in the Bergman spaces to the context of $\Po^2(\mu)$ spaces. It is these applications that are the main motivation for the research presented in this article. 

\subsection{Main research questions}

The measures $\mu$ considered here will be of the form appearing in \eqref{E:MuDef} below. The particular structure of the measure makes $\Po^t(\mu)$ exhibit properties commonly associated with three distinct spaces which have already been mentioned: the Hardy space $\Po^t(d\m)$, the Bergman space $\Po^t(dA)$, and the Lebesgue space $\mathcal{L}^t(d\m)$. The following two questions are answered in the article.

\begin{question}
    What measures $\mu$ of the form \eqref{E:MuDef} make the space $\Po^t(\mu)$ into a genuine space of analytic functions on $\D$? 
\end{question}

\begin{question}
    Given that the space $\Po^t(\mu)$ of the form \eqref{E:MuDef} really is a space of analytic function, what \textit{bounded} functions $f \in \Po^t(\mu)$ are cyclic for the shift operator $M_z: f(z) \mapsto zf(z)$? Namely, what bounded functions are such that $[f]$, the smallest closed $M_z$-invariant subspace of $\Po^t(\mu)$ which contains $f$, is the whole space itself?
\end{question}

The first question will be answered by computing the Thomson decomposition of a measure of the form \eqref{E:MuDef}. The methods which we shall use for this have already been presented in previous articles \cite{bergqvist2024distributing}, \cite{malman2023revisiting} and \cite{malman2023thomson}, but they need to be combined. It is the second question that we consider here to be the important one. Modulo standard arguments, the question can be answered by describing the cyclic \textit{singular inner functions} 
\begin{equation}
    \label{E:SnuEq}
    S_\nu(z) := \exp \Big(-\int_{\T} \frac{\zeta + z}{\zeta - z} d\nu(\zeta)\Big), \quad z \in \D. 
\end{equation}
Here $\nu$ is a finite non-negative Borel measure on $\T$, singular with respect to Lebesgue measure $d\m$. The singular inner functions are members of $\hil^\infty$, the algebra of bounded analytic functions in $\D$, and are characterized as those functions in $\hil^\infty$ which are non-vanishing in $\D$ and have non-tangential boundary values on $\T$ of unit modulus, almost everywhere with respect to the Lebesgue measure $d\m$.

The restriction in the second question to consider bounded functions is important. See Section~\ref{S:RelatedResSec} for comments on this.

\subsection{Beurling-Carleson sets associated to a space} \label{S:BCSetsIntroSec}

We will obtain complete results for measures $\mu$ of the following structure:
\begin{equation}
    \label{E:MuDef} d\mu = dA_\alpha + w \, d\m.
\end{equation} Here 
\begin{equation}
    \label{E:AalphaDef}
    dA_\alpha(z) = (1-|z|)^\alpha dA(z), \quad \alpha > -1
\end{equation} is the standard weighted area measure on the unit disk $\D$, and $w$ is a non-negative Borel measurable function on $\T$. 

Answers to the two questions will be given in terms of a class of subsets of $\T$ associated with the space. Let $h(x): (0,1) \to \R_+$ be the function 
\begin{equation}
    \label{E:CarlesonGaugeDef}
    h(x) = x \log(e/x).
\end{equation} The class of \textit{Beurling-Carleson sets} consists of those closed subsets $E \subset \T$ for which the system $\{\ell_k\}_k$ of maximal open arcs complementary to $E$ in $\T$ satisfies 
\begin{equation}
    \label{E:BChDef}
    \sum_k h(|\ell_k|) < \infty.
\end{equation} Here, and throughout the article, $|S| = m(S)$ is the Lebesgue (arclength) measure of the set $S$. We denote the class of Beurling-Carleson sets by $\BC$, and we emphasize that $\BC$ is not restricted to contain only sets of Lebesgue measure zero, as in some earlier works on cyclic vectors. Furthermore, we introduce the family of \textit{$w$-associated $\BC$ sets}, which is the subclass of $\BC$ on which $w$ is logarithmically integrable: 
\begin{equation} \label{E:AssociatedSetsDef}
    \assoc{w} = \Big\{ E : E \in \BC, \int_E \log w \, d\m > -\infty \Big\}.
\end{equation}
Note that if $E \in \BC$ is of Lebesgue measure zero, then $E \in \assoc{w}$ automatically. 

\subsection{Main results} \label{S:MainResSect}

Our first theorem answers the first of the questions stated above. It is a direct generalization to the weighted context of the main result in \cite{malman2023thomson}. The statement requires some definitions which are generalizations of those in \cite{malman2023thomson}.

Consider the non-negative quantity \[ C := \sup \Big\{|E| : E \in \assoc{w} \Big\}. \] If $C  = 0$, we set $\core{w} = \varnothing$, and otherwise we define \[ \core{w} := \bigcup_n E_n\] where $\{E_n\}_n$ is any increasing sequence of sets in $\assoc{w}$ for which $\lim_n |E_n| = C$. The existence of such a sequence is ensured by the class $\assoc{w}$ being closed under finite unions (see \cite[Proposition 3.2]{malman2023thomson}). It is not hard to see that $\core{w}$ is in this way well-defined up to a set of Lebesgue measure zero, and it does not depend on the initial choice of the sequence $\{E_n\}_n$. Clearly, sets of Lebesgue measure zero in $\assoc{w}$ play no role in shaping $\core{w}$, but they will play a role in our description of cyclic singular inner functions.

We define the residual set of $w$ as the difference of its natural carrier and the core: \[ \res{w} := \{ z \in \T : w(z) > 0 \} \setminus \core{w}.\] Throughout the article, a \textit{carrier} for a Borel measure $\sigma$ on $\T$ will be any subset $C \subset \T$ for which we have $\sigma(C \cap S) = \sigma(S)$ for all Borel subsets $S$ of $\T$. If $w \in \mathcal{L}^1(d\m)$, then we say that $C$ is a carrier for $w$ if it is a carrier for the measure $w\, d\m$.

Like $\core{w}$, the set $\res{w}$ is also defined only up to a set of Lebesgue measure zero. We decompose $w$ into pieces $w_c$, $w_r$ living on the core and residual, respectively:
\[ w_c = w|\core{w}, \quad w_r = w|\res{w}.\]

In the considered family of $\Po^t(\mu)$ spaces corresponding to measures $\mu$ of the form \eqref{E:MuDef}, the restriction $f|\D$ of any element $f \in \Po^t(\mu)$ is an analytic function living in a Bergman space $\Po^t(dA_\alpha)$. We will say that the space $\Po^t(\mu)$ of the form \eqref{E:MuDef} is \textit{irreducible} if the restriction mapping $f \mapsto f|\D$ is injective on $\Po^t(\mu)$. In such a case, each element $f \in \Po^t(\mu)$ can be uniquely identified with an analytic function on $\D$, and so $\Po^t(\mu)$ is itself a space of analytic functions on $\D$, continuously contained in the Bergman space $\Po^t(dA_\alpha)$. Although it may not be a priori obvious, our definition of irreducibility for our class of spaces coincides with the more general one in \cite{aleman2009nontangential} and \cite{thomson1991approximation} which involves non-existence of non-trivial indicator functions in the space. We will sometimes use the term \textit{analytic $\Po^t(\mu)$ space} synonymously with \textit{irreducible $\Po^t(\mu)$-space}.

\begin{mainthm}
    \thlabel{T:ThomsonDecompTheorem}
    Let $\mu$ be as in \eqref{E:MuDef}. We have the decomposition
    \[ \Po^t(\mu) = \Po^t( dA_\alpha + w_c \, d\m) \oplus \mathcal{L}^t(w_r \, d\m),\] where $\Po^t( dA_\alpha + w_c \, d\m)$ is irreducible.
\end{mainthm}

More precisely, the decomposition is to be interpreted in the following sense: any $f$ in one of the summands on the right-hand side has an extension to an element $\mathcal{L}^t(\mu)$ by setting $f$ to zero outside of its initial carrier set, and these extended elements are in fact members of $\Po^t(\mu)$ and they span the whole space. The theorem is a refined version of a result from \cite{bergqvist2024distributing}, which established the decomposition in an important special case $w = w_r$, and which confirmed an old conjecture of Kriete and MacCluer from \cite{kriete1990mean}. The irreducibility of the piece $\Po^t( dA_\alpha + w_c \, d\m)$ has essentially been known since the 1970s as a consequence of the work of Khrushchev in \cite{khrushchev1978problem}, although his results are not coined in terms of $\Po^t(\mu)$ spaces. We will fill in the details of the necessary modifications of his proof in Section \ref{S:ThomsonDecompProofSec}. A different proof of irreducibility for $t=2$ is given in \cite[Proposition 5.1]{limani2024constructions}. Our principal contribution in this direction establishes that $\Po^t( dA_\alpha + w_c \, d\m)$ is the maximal irreducible piece. We deduce from \thref{T:ThomsonDecompTheorem} that our $\Po^t(\mu)$ is a space of analytic functions on $\D$ if and only if $\core{w}$ is a carrier set for $w$.

Having settled the analyticity question, we may get to the main matter. Here is our cyclicity result.

\begin{mainthm}
    \thlabel{T:CyclicityMainTheorem}
    Let $\mu$ be as in \eqref{E:MuDef}, and assume that $\Po^t(\mu)$ is a space of analytic functions on $\D$. The singular inner function $S_\nu$ is cyclic in $\Po^t(\mu)$ if and only if $\nu$ vanishes on all associated sets:  
    \[\nu(E) = 0, \quad E \in \assoc{w}.\]
\end{mainthm}

It is the sufficiency for cyclicity of the vanishing condition on the associated sets that should be seen as the main new contribution of this article. The necessity of the condition is easier to establish than the sufficiency, and one can say a bit more. In \cite[Corollary 6.5]{limani2024constructions} it has been proved that if $\nu$ is supported on a set in $\assoc{w}$ and $[S_\nu]$ is the $M_z$-invariant subspace generated by $S_\nu$ in $\Po^2(\mu)$, then we have $f/S_\nu \in \hil^\infty$ for any $f \in [S_\nu] \cap \hil^\infty$. In other words, $S_\nu$ appears as a factor in the usual inner-outer factorization of any non-zero $f \in [S_\nu] \cap \hil^\infty$. This property of course implies that $1 \not\in [S_\nu]$, so such $S_\nu$ is not cyclic. In Section \ref{S:PermanenceSubsec} we use methods from \cite{berman1984cyclic} to extend this result to $t \neq 2$. 

Using the result indicated in the above paragraph, we may describe the $M_z$-invariant subspace $[f]$ of $\Po^t(\mu)$ generated by a bounded function $f$ in the following way. Let $f = BS_\nu U$ be the inner-outer factorization of $f$ into a Blaschke product $B$, singular inner function $S_\nu$, and a bounded outer function $U$. Let $\nu_p$ be the least upper bound of all restrictions $\nu|E$ of $\nu$ to sets $E \in \assoc{w}$, and write 
\begin{equation}
    \label{E:NuDecomp}
    \nu = \nu_p + \nu_c.
\end{equation} 

One way to construct $\nu_p$ more explicitly is to employ an argument similar to the one used in the construction of the set $\core{w}$. Namely, take an increasing sequence $\{E_n\}_n$ of sets in $\assoc{w}$ which satisfies \[ \lim_{n \to \infty} \nu(E_n) = \sup_{E \in \assoc{w}} \nu(E) \] and define $\nu_p$ as the restriction of $\nu$ to the Borel set $\bigcup_n E_n$. Clearly $\nu_p$ constructed in this way satisfies the measure inequality $\nu|E \leq \nu_p$ for any $E \in \assoc{w}$. We note that $\nu_p$ has as a carrier a countable union of sets in $\assoc{w}$, and $\nu_c(E) = 0$ for $E \in \assoc{w}$. From \thref{T:CyclicityMainTheorem} we can deduce the following consequence. 

\begin{cor*}
    Let $f = BS_\nu U$ be the inner-outer factorization of a function $f \in \hil^\infty$. Then \[[f] = [BS_{\nu_p}],\] and every function $h \in [f] \cap \hil^\infty$ satisfies $h/BS_{\nu_p} \in \hil^\infty$.
\end{cor*}

We give a proof of the corollary at the end of Section~\ref{S:CyclicitySection}. \thref{T:CyclicityMainTheorem} and the corollary answer the second of our questions stated above. 

We remark that there is a huge difference between the trivial case $\T \in \assoc{w}$, which implies that $|\T \setminus \core{w}| = 0$, and this latter condition $|\T \setminus \core{w}| = 0$. In the first case we have $\log w \in \mathcal{L}^1(d\m)$, and the structure of $\Po^t(\mu)$ is simple. It is then a space of analytic functions contractively contained in a weighted Hardy space $W \cdot \hil^t$, where $W$ is an outer function satisfying $|W| = w^{-1/2}$ on $\T$, and where the norm of $W \cdot f$ in $W \cdot \hil^t$ is the norm of $f$ in $\hil^t$. In particular, no singular inner functions are then cyclic in $\Po^t(\mu)$. However, in the second case, such an identification does not in general hold. Although we shall not go into details of a construction, we wish to mention that it is possible to exhibit examples of singular inner functions $S_\nu$ which are cyclic in a space $\Po^t(\mu)$ constructed from $\mu$ of the kind \eqref{E:MuDef} for which $|\T \setminus \core{w}| = 0$.

\subsection{Related results and comments}

\label{S:RelatedResSec}

\subsubsection{Korenblum-Roberts Theorem}

A singular inner function is not cyclic in the Hardy spaces $\hil^t = \Po^t(d\m)$. The situation is different in the context of Bergman spaces $\Po^t(dA_\alpha)$, and was explained in the 1980s in the works of Korenblum in \cite{korenblum1981cyclic} and Roberts in \cite{roberts1985cyclic}. They gave two very different proofs of the following statement, a precursor of our main result in \thref{T:CyclicityMainTheorem}.

\begin{thm*}[\textbf{Korenblum-Roberts cyclicity theorem}] The singular inner function $S_\nu$ is cyclic in the standard weighted Bergman spaces $\Po^t(dA_\alpha)$, $t\in (0,\infty)$, if and only if $\nu$ vanishes on all Beurling-Carleson sets of Lebesgue measure zero: 
\[ \nu(E) = 0, \quad E \in \BC, \, |E| = 0.\]
\end{thm*}

Korenblum derives the result from his famous works \cite{korenblum1975extension} and \cite{korenblum1977beurling} on cyclic vector in the so-called growth classes. The proof of Roberts is completely different and of importance in its own right. Note that if we set $w \equiv 0$ on $\T$, then $\assoc{w}$ coincides with the family of $\BC$ sets of Lebesgue measure zero. Consequently, the Korenblum-Roberts theorem is a special case of \thref{T:CyclicityMainTheorem}, but the argument we present below is not a new proof of the classical result. To establish \thref{T:CyclicityMainTheorem}, we will use extensions of linear programming techniques of Korenblum from \cite{korenblum1975extension}, \cite{korenblum1977beurling} augumented with recently obtained specialized constrained $\mathcal{L}^1$-optimization tools from \cite{bergqvist2024distributing} which we describe in Section \ref{S:HausdorffFuncSec}. Techniques of Roberts, on the other hand, have recently found interesting applications, for instance in the works of Ivrii in \cite{ivrii2019} and Ivrii and Nicolau in \cite{ivrii2022beurling}. 

\subsubsection{Faster decreasing weights on $\D$}

Several of our results have versions for measures $\mu$ of structure similar to \eqref{E:MuDef} but where $dA_\alpha$ is replaced by a radial measure $G(|z|)dA(z)$ with $G(x) \to 0$ as $x \to 1$ sufficiently slowly, see \eqref{E:MuGeneralizedDef} below. These results have essentially the same proofs as our main theorems, and in the coming sections we shall carry out our arguments in the general context and state the generalized results. Our method does not apply to weights of the form $G(x) = \exp\big(-c/(1-x)\big)$ for $c > 0$. The structure of the space $\Po^2(\mu)$ of this form has been studied in \cite{malman2023revisiting} and \cite{malman2023shift}, and is analogous to the one presented here in \thref{T:ThomsonDecompTheorem} and \thref{T:CyclicityMainTheorem}, but where the associated sets $\assoc{w}$ replaced by the family of intervals on which $w$ has an integrable logarithm. The new difficulties in the setting considered here are related to the complexity of the sets in $\assoc{w}$ in comparison to intervals. In particular, in \cite{malman2023revisiting} and \cite{malman2023shift} no use of linear optimization theory was necessary.

\subsubsection{The work of Aleman, Richter and Sundberg}
Work on analytic $\Po^t(\mu)$ spaces from a different viewpoint than ours has been carried out by Aleman, Richter and Sundberg. Their article \cite{aleman2009nontangential} is concerned with function and operator theory in $\Po^t(\mu)$ spaces for a general measure $\mu$ supported in the closed unit disk $\cD = \D \cup \T$. In the article, the authors work from the get go under the assumption that $\Po^t(\mu)$ is a space of analytic functions on $\D$, without any additional assumptions on the structure of the measure $\mu$. Using in part methods of Thomson from \cite{thomson1991approximation}, they obtain the important existence result for non-tangential boundary values for functions in the space, and as a consequence obtain information regarding zero sets and interpolating sequences. They prove also the important index theorem for $M_z$-invariant subspaces: given any $M_z$-invariant subspace $\mathcal{N} \subset \Po^t(\mu)$, the orthogonal complement of $M_z\mathcal{N}$ inside $\mathcal{N}$ has dimension $1$. 


\subsubsection{Unbounded cyclic functions}
We stress that the restriction to consider bounded functions $f \in \hil^\infty$ in our cyclicity problem is motivated by the applications, which we shall soon present, but it is also critical. Our main result in \thref{T:CyclicityMainTheorem} and the corollary above has a straight-forward extension to cyclic functions $f$ in the Nevanlinna class, i.e, quotients $f = d/c$ of bounded analytic functions $d$, $c$ in $\D$ (see the end of Section~\ref{S:CyclicitySection} for the precise statement and a short argument), but the problem of characterizing the cyclicity for a more general function $f \in \Po^t(\mu)$, even in the case $d\mu = dA$, is a well-known open problem in the theory of Bergman spaces. In that context, a sufficient condition and a necessary condition, which are not too far apart, follow from Korenblum's works \cite{korenblum1975extension}, \cite{korenblum1977beurling} on the \textit{extended Nevanlinna class}. Widening of our results to the setting of Korenblum's works requires some further efforts which are not part of this article.

\subsection{An application to $\hb$ spaces}

Our corollary to \thref{T:CyclicityMainTheorem} has a direct application to the theory of de Branges-Rovnyak spaces $\hb$, a family of Hilbert spaces of analytic functions in $\D$ associated to analytic self-maps $b$ of $\D$. We refer the reader to \cite{hbspaces1fricainmashreghi}, \cite{hbspaces2fricainmashreghi} and \cite{sarasonbook} for the theory of these spaces. With aid of the results in \cite{limani2023problem} and \cite{limani2024constructions}, results of the present work enable us to complete the characterization of the class of symbols $b$ corresponding to spaces $\hb$ admitting a norm dense subset of functions in $\A^\infty$, the algebra of analytic functions in $\D$ with $C^\infty$ extensions to the boundary $\T$.

It was found in \cite[Theorem 1.1]{limani2023problem} that two necessary conditions for the density of $\A^\infty \cap \hb$ in $\hb$ can be expressed in terms of the structure of $M_z$-invariant subspaces of $\Po^2(\mu)$, where $\mu$ has the form \eqref{E:MuDef}, and
\begin{equation}
\label{E:DeltaWeightDef}
w = \Delta := \sqrt{1-|b|^2}.
\end{equation} The parameter $\alpha$ in the definition of $dA_\alpha$ is unimportant. If $b = BS_\nu U$ is the inner-outer factorization of $b$ into a Blaschke product $B$, singular inner functions $S_\nu$ and outer function $U$, then the two necessary conditions, expressed in the language of this article, are that
\begin{enumerate}[(i)]
    \item $\Po^2(\mu)$ should be a space of analytic functions,
    \item any function $h \in [S_\nu] \cap \hil^\infty$ should satisfy $h/S_\nu \in \hil^\infty$.
\end{enumerate}

The results of this article imply that (i) is satisfied if and only if the set $\core{\Delta}$ is a carrier for $\Delta$, while (ii) is satisfied if and only if $\nu = \nu_p$ in \eqref{E:NuDecomp}. Moreover, in \cite[Theorem B]{limani2024constructions} it is shown that these last two conditions are sufficient for density of $\A^\infty \cap \hb$ in $\hb$. Combining these results, we obtain a complete and explicit characterization.

\begin{mainthm} \thlabel{T:HbDensityTheorem}
    Let $b = BS_\nu U$ be the inner-outer factorization of $b$, and let $\Delta$ be given by \eqref{E:DeltaWeightDef}. The set $\A^\infty \cap \hb$ is norm dense in the space $\hb$ if and only if the following two conditions are satisfied:
    \begin{enumerate}[(i)]
        \item the set $\core{\Delta}$ is a carrier for $\Delta$, 
        \item we have $\nu = \nu_p$ in \eqref{E:NuDecomp}.
    \end{enumerate}
\end{mainthm}

In other words, the two conditions say that both $\Delta$ and $\nu$ live on a countable union of sets in $\assoc{\Delta}$. The union corresponding to $\nu$ may need to be a blend of $\BC$-sets of both zero and positive Lebesgue measure. Note that Sarason's classical condition $\log \Delta \in \mathcal{L}^1(\T)$, which appears for instance in \cite{sarasonbook} and characterizes the density of the set of analytic polynomials $\Po$ in $\hb$, simply means that $\T \in \assoc{\Delta}$. In this case, conditions $(i)$ and $(ii)$ in \thref{T:HbDensityTheorem} are obviously satisfied.

\thref{T:HbDensityTheorem} can be seen as a companion to a previous result of the author and Aleman in \cite{comptesrenduscont}, in which it is asserted that functions in $\hb$ continuous up to the boundary $\T$ are always dense in $\hb$. 

\subsection{Outline of the paper}

In Section \ref{S:HausdorffFuncSec}, we introduce our principal optimization tools from \cite{bergqvist2024distributing} which form the technical backbone of the article. In Section \ref{S:ThomsonDecompProofSec} we use the results of the previous section to establish \thref{T:ThomsonDecompTheorem}. Section \ref{S:PermanenceSubsec} is concerned with establishing the necessity for cyclicity of the vanishing condition on associated sets in \thref{T:CyclicityMainTheorem}. The proof of \thref{T:CyclicityMainTheorem} is completed in Section \ref{S:CyclicitySection}. There, we combine the optimization tools from Section \ref{S:HausdorffFuncSec} with Korenblum's linear programming techniques from \cite{korenblum1975extension} and \cite{korenblum1977beurling} to establish the sufficiency for cyclicity of the vanishing condition in \thref{T:CyclicityMainTheorem}. Section~\ref{S:CyclicitySection} contains also proofs of some other auxiliary results mentioned above.

Mainly for convenience, we follow Korenblum's idea to use the concept of a premeasure. With a reader unfamiliar with this concept in mind, and because of our need for certain simple generalizations of Korenblum's premeasure results, we include Appendix \ref{sec:appendixA} which covers the basic facts regarding premeasures that are used in the proofs.

\section{An optimization problem and the Hausdorff functional} \label{S:HausdorffFuncSec}

In this section, we introduce our principal $\mathcal{L}^1$-optimization problem and the functionals $\mathcal{M}(h,R)$ from \cite{bergqvist2024distributing}. We derive the critical lower estimate in \eqref{E:MainEstimateResidualBound} on the optimal value in the optimization problem from a duality theorem relating it to the functional $\mathcal{M}(h,R)$. 

In the context of our main results, we think of $h$ as a positive scalar multiple of \eqref{E:CarlesonGaugeDef}, but the functional $\mathcal{M}(h,R)$ may be defined on a larger class of functions $h$ which we will introduce next. The broader context leads to no additional difficulties in the proofs, and on the upside it allows us to prove results slighly more general than the ones presented in \thref{T:ThomsonDecompTheorem} and \thref{T:CyclicityMainTheorem}.

\subsection{The principal $\mathcal{L}^1$-optimization problem}  \label{S:OptimProbSubsec}

The most important properties of the function $h$ in \eqref{E:CarlesonGaugeDef} are that $h$ is increasing and continuous on $(0,1]$, satisfies $h(0) := \lim_{x \to 0} h(x) = 0$, and additionally
\begin{enumerate}[(R1)]
    \item $h(x)/x$ is decreasing in $x$,
    \item $\lim_{x \to 0} h(x)/x = \infty$.
\end{enumerate} We will say that any function $h$ is a \textit{gauge function} if it satisfies the above properties. It is easy to verify subadditive inequality $h(x+y) \leq h(x) + h(y)$. The argument $x$ of $h(x)$ will for the most part be the length $|I|$ of an interval $I \subset \T$. We use throughout the article the normalization $|\T| = m(\T) = 1$. 

Given a non-negative function $R:\T \to [0,\infty]$ and a gauge function $h$, the family $\mathcal{F}(h,R)$ is to consist of all non-negative functions on $\T$ dominated pointwise by $R$ and satisfying a local mass distribution bound defined in terms of the gauge $h$:
\begin{equation}
    \label{E:FFamilyDef}
    \mathcal{F}(h,R) := \Big\{ f \in \mathcal{L}^1(d\m) : 0 \leq f \leq R, \int_\ell f \,d\m \leq h(|\ell|) \text{ for all intervals } \ell \Big\}.
\end{equation}

Denote by $\|f\|_1 = \int_\T |f| \, d\,m$ the usual $\mathcal{L}^1(d\m)$ norm of $f$. The optimal amount of mass that can be distributed under the constraints defining $\mathcal{F}(h,R)$, namely
\[ \sup \, \Big\{ \|f\|_1: f \in \mathcal{F}(h,R) \Big\},\] will be of main importance in the article.

\subsection{Hausdorff functionals and their duality}

Here is a dual way to express the previous supremum. With notation as above, we intoduce the quantity
\begin{equation}
    \label{E:HausdorffFuncDef}
    \mathcal{M}(h,R) := \inf_{ \mathcal{U}} \, \Big( \sum_{_{\ell \in \mathcal{U}}} h(|\ell|) + \int_{\T \setminus \cup_{\ell \in \mathcal{U}} \ell} R \, d\m \Big)
\end{equation} where the infimum above is taken over all families $\mathcal{U} = \{\ell\}$ of open intervals $\ell$ in $\T$. With $h$ fixed and $R = \infty \cdot 1_E$, where $1_E$ is the indicator function of a measurable subset $E \subset \T$, the mapping $E \mapsto \mathcal{M}(h,\infty \cdot 1_E)$ is similar to a degree to the definition of the usual Hausdorff content of the set $E$. We will call $\mathcal{M}(h,R)$ for the \textit{Hausdorff functional}. 

The importance of the Hausdorff functional comes from the following result established in \cite{bergqvist2024distributing}. It is of the type commonly encountered in the duality theory for linear programs.

\begin{prop}{\textbf{(Duality for Hausdorff Functionals)}}
    \thlabel{P:DualityThm}
    We have the inequality
    \[ 6 \cdot \sup \Big\{ \|f\|_1: f \in \mathcal{F}(h,R) \Big\} > \,  \mathcal{M}(h,R).\]
\end{prop}

We note that the reverse inequality \[\sup \Big\{ \|f\|_1: f \in \mathcal{F}(h,R) \Big\} \leq \,  \mathcal{M}(h,R)\] follows readily from the definitions (see \cite{bergqvist2024distributing}). 

\subsection{Generalized cores, residuals, and the residual bound}
\label{S:GenCoreResSubsec}

We may define the family $\assoch{h}{w}$ and the sets $\coreh{h}{w}$, $\resh{h}{w}$ analogously to how it was done in the Introduction by replacing the function in \eqref{E:CarlesonGaugeDef} with a more general gauge function. Namely, for any gauge function $h$, we define the family $\BCh$ of closed subsets of $\T$ which satisfy \eqref{E:BChDef}, and we let $\assoch{h}{w}$ be defined as in \eqref{E:AssociatedSetsDef} with $\BCh$ replacing $\BC$. The sets $\coreh{h}{w}$ and $\resh{h}{w}$ are also defined analogously. We set 
\[ \coreh{h}{w} = \bigcup_n E_n\] for any increasing sequence $\{E_n\}_n$ of $\BCh$ sets satisfying \[\lim_n |E_n| = \sup \Big\{|E| : E \in \assoch{h}{w} \Big\},\] and we set \[ \resh{h}{w} =  \{ z \in \T : w(z) > 0 \} \setminus \coreh{h}{w}.\] For $h$ in \eqref{E:CarlesonGaugeDef}, our generalized definitions reduce to the ones stated in the Introduction.

Properties of core and residual sets in the unweighted context are presented in \cite{malman2023thomson}. With some necessary and natural modifications, these properties carry over to the weighted context. The most important property of the set $\coreh{h}{w}$ is the implication 
\begin{equation}
    \label{E:EAssocCoreImplication}
    E \in \assoch{h}{w} \quad \Rightarrow \quad |E \setminus \coreh{h}{w}| = 0. 
\end{equation}
In other words, up to differences of Lebesgue measure zero, the set $\coreh{h}{w}$ contains all sets in $\assoch{h}{w}$. The implication \eqref{E:EAssocCoreImplication} is non-trivial only for sets $E$ of positive Lebesgue measure, of course. The proof of the implication is straight-forward (see the proof of \cite[Proposition 3.3]{malman2023thomson} for a similar argument).

The following weighted version of \cite[Proposition 3.5]{malman2023thomson} is the most important property of the set $\resh{h}{w}$ which we shall use.

\begin{lem}{\textbf{(Residual Lower Bound)}}
    \thlabel{L:ResBoundLemma} Let $R = \log^+(1/w)$. If $I$ is an interval in $\T$ and $R|I$ is the restriction of $R$ to $I$, then \[ \mathcal{M}(h, R|I) \geq h(|I \cap \resh{h}{w}|).\]
\end{lem}

In the above statement, and throughout the article, we use a slightly non-standard convention: we interpret the restriction $f|S$ of a function $f$ to a set $S$ as the function coinciding with $f$ on the set $S$, and vanishing elsewhere. As usual, we set $\log^+(x) := \max( \log x, 0)$. Note that we have the equivalence \[ \int_E \log^+(1/w)\, d\m < +\infty \quad  \Leftrightarrow \quad \int_E \log w \, d\m > -\infty. \] In other words, $E \in \BCh$ is in $\assoch{h}{w}$ if and only if $R$ is integrable on $E$.

\begin{proof}[Proof of \thref{L:ResBoundLemma}] Let $\mathcal{U} = \{\ell\}$ be any family of open intervals contained in $I$ for which the right-hand side in \eqref{E:HausdorffFuncDef} is finite (with $R$ replaced by $R|I$). Setting $E := I \setminus \cup_{\ell \in \mathcal{U}} \ell$, it follows that 
\begin{equation}
\label{E:EAssocIneq}
\int_E R \, d\m = \int_E \log^+(1/w) \, d\m < \infty.
\end{equation} If we assume that $I$ is closed, which we clearly may without loss of generality, then $E$ is closed. Since $\sum_{\ell \in \mathcal{U}} h(|\ell|) < \infty$, it follows that $E \in \BCh$, and \eqref{E:EAssocIneq} shows that in fact we have $E \in \assoch{h}{w}$. This implies by \eqref{E:EAssocCoreImplication} that $|E \setminus \coreh{h}{w}| = 0$, and so $E$ is contained in $\coreh{h}{w}$, up to a set of Lebesgue measure zero. In particular, we must have $|E \cap \resh{h}{w}| = 0$, from which it follows that $I \cap \resh{h}{w}$ is contained in $\cup_{\ell \in \mathcal{U}} \ell$, again up to a set of Lebesgue measure zero. But then \[ |I \cap \resh{h}{w}| \leq |\cup_{\ell \in \mathcal{U}} \ell| \leq \sum_{_{\ell \in \mathcal{U}}} |\ell|,\] and, since $h$ is increasing and subadditive, we obtain
\[ h(|I \cap \resh{h}{w}|) \leq \sum_{_{\ell \in \mathcal{U}}} h(|\ell|) \leq \sum_{_{\ell \in \mathcal{U}}} h(|\ell|) + \int_{I \setminus \cup_{\ell \in \mathcal{U}} \ell} R \, d\m.\] The claim follows by taking infimum over families $\mathcal{U} = \{\ell\}$ in the last inequality.
\end{proof}

Combining \thref{P:DualityThm} with \thref{L:ResBoundLemma}, and using notation as in \thref{L:ResBoundLemma}, we obtain the important lower bound

\begin{equation}
    \label{E:MainEstimateResidualBound}
    \sup \Big\{ \|f\|_1: f \in \mathcal{F}(h,R|I) \Big\} > \frac{h(|I \cap \resh{h}{w}|)}{6}
\end{equation}

This critical estimate will let us generalize the main result of \cite{malman2023thomson} into \thref{T:ThomsonDecompTheorem}. Note that a truncation shows that there always exist a \textit{bounded} function $f \in \mathcal{F}(h,R)$ for which $6 \int_\T f\, d\m = \mathcal{M}(h,R)$.

\section{Thomson decomposition of the measure}

\label{S:ThomsonDecompProofSec}

The goal of the section is to prove \thref{T:ThomsonDecompTheorem}. At no additional strain, we carry out the proof in the context of general gauge function $h$ and in this way obtain more general results. We assume throughout that $\log w \not\in \mathcal{L}^1(d\m)$, since the contrary case is trivial, as explained at the end of Section \ref{S:MainResSect}.

\subsection{A more general class of measures}

Given a gauge function $h$ satisfying the properties (R1) and (R2) stated in Section \ref{S:OptimProbSubsec}, we may associate to it the the family of measures 
\begin{equation}
    \label{E:MuGeneralizedDef}
    d\mu = G_{a h} dA + w \, d\m, \quad a > 0,
\end{equation} where \[ G_{a h}(z) = \exp \Big(-a \frac{h(1-|z|)}{1-|z|} \Big), \quad z \in \D.\]

Note that the functions $G_{a h}(x)$ are decreasing in $x \in (0,1)$, and so strictly speaking \eqref{E:MuGeneralizedDef} does not include measures of the form \eqref{E:MuDef} for $\alpha \in (-1,0)$, but this difference is insignificant. For certain choices of the gauge $h$ a direct generalization of \thref{T:ThomsonDecompTheorem} to measures of the form \eqref{E:MuGeneralizedDef} holds. The corresponding decomposition of the weight $w$ is 
\[ w_c = w|\coreh{h}{w}, \quad w_r = w|\resh{h}{w}\] with definitions as in Section \ref{S:GenCoreResSubsec}. 
More precisely, we will show that for any $h$ we have the decomposition
\begin{equation}
    \label{E:GeneralizedThomsonDecomp}
    \Po^t(\mu) = \Po^t( G_{a h}dA + w_c \, d\m) \oplus \mathcal{L}^t(w_r \, d\m).
\end{equation} The irreducibility of the piece $\Po^t( G_{a h} + w_c \, d\m)$ will be established in Section \ref{S:IrreducibilityProofSec} below, but only under some additional regularity assumptions on $h$ which in particular hold for the gauge function in \eqref{E:CarlesonGaugeDef}.

In context of the measures $\mu$ of the form \eqref{E:MuDef}, we always assume that the gauge $h$ has the form \eqref{E:CarlesonGaugeDef}.

\subsection{A sufficient condition for establishing the decomposition}

To establish the direct sum decomposition in \thref{T:ThomsonDecompTheorem}, or the one in \eqref{E:GeneralizedThomsonDecomp}, it will suffice to show that
\begin{equation}
    \label{E:L2wrContP2mu}
    \mathcal{L}^t(w_r\, d\m) \subset \Po^t(\mu).
\end{equation} Here we interpret $f \in \mathcal{L}^t(w_r\, d\m)$ as an element of $\mathcal{L}^t(\mu)$ by extending $f$ to be zero outside of the set $\{ z \in \T : w_r(z) > 0 \}$ (which is well-defined up to a set of Lebesgue measure zero). Indeed, if \eqref{E:L2wrContP2mu} holds, then for every $f \in \Po^t(\mu)$ we have that $f - f|\resh{h}{w} \in \Po^t(dA_\alpha + w_c \, d\m) \cap \Po^t(\mu)$, and $f| \resh{h}{w} \in \mathcal{L}^t(w_r \, d\m) \cap \Po^t(\mu)$. 

The following lemma, essentially contained in \cite[Lemma 2.1 and proof of Lemma 3.1]{malman2023revisiting}, lets us reduce our task to a construction of a specific function inside $\Po^t(\mu)$.

\begin{lem} \thlabel{L:ThompsonDecompReductionLemma}
In order to establish \eqref{E:L2wrContP2mu}, it suffices to show that there exists $F \in \Po^t(\mu)$ which satisfies $F|\D \equiv 1$, $F|\resh{h}{w} \equiv 0$, and $F|\T \in \mathcal{L}^{t^*}(w \,d\m)$ for some $t^* > \max(t,1)$.
\end{lem}

\begin{proof}
If $f \in \Po^t(\mu)$ is as stated, then $g = 1 - F \in \Po^t(\mu)$ vanishes on $\D$ and is non-zero almost everywhere on $\resh{h}{w}$. The $M_z$-invariant subspace $[g]$ of $\Po^t(\mu)$ generated by $g$ is then a subspace of $\mathcal{L}^t(w\, d\m)$, with containment interpreted as in \eqref{E:L2wrContP2mu}.

Let us first assume that $t > 1$. Then, since $[g]$ is a subspace of $\mathcal{L}^t(w \, d\m)$ which is invariant for the operator $M_z$, the space \[ w^{1/t}[g] := \{ w^{1/t} h : h \in [g] \}\] is contained in $\mathcal{L}^t(d\m)$ and it has one of the two forms well known from the classical Beurling-Wiener Theorem (see \cite[Section 2.3]{malman2023revisiting}). Namely, either we have 
\begin{equation}
\label{E:BeurlingSubspace}
w^{1/t}[g] = \{ q H : H \in \hil^t \}
\end{equation} for some unimodular function $q$ on $\T$, $\hil^t$ being the Hardy space, or for some $S \subset \T$ we have
\begin{equation}
\label{E:WienerSubspace}
w^{1/t}[g] = \mathcal{L}^t(d\m|S) = \{f \in \mathcal{L}^t(d\m) : f \equiv 0 \text{ \m-a.e outside of } S \}.
\end{equation}
In the first case, every non-zero function $d \in [g]$ can be expressed as $w^{1/t}d = qH$ for some non-zero $H \in \hil^t$, and so 
\[ \int_\T \log (w^{1/t}|d|) \, dm = \int_\T \log |H| \, d\m > -\infty,\] the last inequality being a well-known property of non-zero functions in $\hil^t$. Since $g \in \mathcal{L}^{t^*}(w \, d\m)$, $t^* > t$, a simple computation shows that \[ \int_\T \log(w^{1/t}|g|) d\m =-\infty\] (see \cite[proof of Lemma 3.1]{malman2023revisiting} and recall that we assume $\log w \not \in \mathcal{L}^1(d\m)$). Since we can set $d = g$ above, the alternative \eqref{E:BeurlingSubspace} is excluded, and we deduce that we are in the situation in \eqref{E:WienerSubspace}. Since  $w^{1/t}g$ does not vanish on the set $\resh{h}{w}$, the set $S$ in \eqref{E:WienerSubspace} must contain $\resh{h}{w}$, up to a difference of Lebesgue measure zero. It follows that if $f \in \mathcal{L}^t(w_r \, d\m)$, then \[w^{1/t}f \in \mathcal{L}^t(d\m|S) = w^{1/t}[g] \] Hence $\mathcal{L}^t( w_r d\m) \subset [g]$, and we have verified \eqref{E:L2wrContP2mu} in the case $t > 1$.

If $t \leq 1$, then since for $t < t'$, convergence in $\mathcal{L}^{t'}(w \, d\m)$ implies convergence in $\mathcal{L}^t(w \, d\m)$, the previous argument may be applied to some $t' > 1$ slightly smaller than $t^*$ to conclude that $[g]$ (which, we emphasize, is a subspace of $\Po^t(\mu)$) contains the corresponding $M_z$-invariant subspace generated by $g$ in $L^{t'}(w \, d\m)$. By the proof of the case $t > 1$, in particular $[g]$ will contain all bounded measurable functions living on the set where $w_r$ is non-zero. A density argument then shows that $[g]$ contains $L^t(w_r \, d\m)$.
\end{proof}

\subsection{The construction} \label{S:ConstrSubsecSplitting}

We proceed to show how to obtain $f \in \Po^t(\mu)$ satisfying the properties mentioned in \thref{L:ThompsonDecompReductionLemma}. The function $f$ will be obtained as a weak cluster point of the outer functions 
\begin{equation}
    \label{E:BigF_NDef}
    F_N(z) := \exp \Big( \int_{\T} \frac{\zeta + z}{\zeta - z} f_N(\zeta) d\m(\zeta) \Big), \quad z \in \D, N \in \mathbb{N}
\end{equation} where $f_N \in \mathcal{L}^\infty(d\m)$ are carefully chosen real-valued functions. Note that $F_N \in \hil^\infty$, and so the containment $F_N \in \Po^t(\mu)$ follows from a straight-forward argument involving a dilation and a Taylor series truncation. 

The following result is the improvement of \cite[Lemma 3.1]{bergqvist2024distributing} needed to establish the sought-after direct sum decomposition. The proof is the same, with the exception that we utilize the new inequality \eqref{E:MainEstimateResidualBound}. 

\begin{lem}[\textbf{Main Construction}]
\thlabel{L:ConstructionLemma}
Let $h$ be a gauge function. There exists a sequence of real-valued function $\{f_N\}_N$ on $\T$ satisfying the following conditions:
\begin{enumerate}[(i)]
    \item for any $\epsilon > 0$, the pointwise inequality $f_N \leq \epsilon \cdot \log^+(1/w)$ holds for all $N$ sufficiently large,
    \item for any $\epsilon > 0$, the inequality $\int_I f_N \, d\m \leq \epsilon \cdot h(|I|)$ holds for all intervals $I \subset \T$ and all $N$ sufficiently large,
    \item $\int_\T f_N \, d\m = 0$,
    \item $f_N \leq 0$ on $\resh{h}{w}$, and $f_N(x) \to -\infty$ $\m$-almost everywhere on $\resh{h}{w}$.
\end{enumerate}    
\end{lem}

\begin{proof}
Fix a positive integer $N$ and set \begin{equation}
    \label{E:ANdef} A_N := \{ \zeta \in \resh{h}{w} : R(\zeta) = \log^+\big(1/w(\zeta)\big) \leq M(N)\},  
\end{equation} where $M(N)$ is some positive number soon to be specified. Note that if $M(N) \to \infty$ as $N \to \infty$, then $\cup_{N} A_N$ is a set of full measure in $\resh{h}{w}$, namely $|\cup_N A_N| = |\resh{h}{w}|$. This will be important at the end of the proof, when part $(iv)$ above will be verified for our construction. 

Denote by $1_A$ the indicator function of a set $A$. Divide $\T$ into $N$ half-open intervals $\{I_k\}_{k=1}^N$ of equal length, and define $f_N$ according to 
\begin{equation}
    \label{E:f_Ndef}
    b(N) \cdot f_N := \sum_{k : |I_k \cap A_N| > 0} g_{N,k} \cdot 1_{I_k \setminus A_N}  - C_{N,k} \cdot 1_{I_k \cap A_N},
\end{equation} where $b(N)$ is a large positive number soon to be specified, where $g_{N,k} \in \mathcal{F}(h,R|I_k)$ is bounded and satisfies 
\begin{equation}
    \label{E:GNkLowerMassEst}
    \int_{I_k} g_{N,k} \, d\m \geq \frac{h(|I_k \cap \resh{h}{w}|)}{6} \end{equation}
and where $C_{N,k}$ is the unique positive number which ensures that $\int_{I_k} f_N \, d\m = 0$, that is, 
\begin{equation}
    \label{E:CNkdef}
    C_{N,k} = \frac{\int_{I_k\setminus A_N} g_{N,k} \, d\m}{|I_k \cap A_N|}.
\end{equation} The existence of $g_{N,k}$ is ensured by inequality the \eqref{E:MainEstimateResidualBound}. 

Since $g_{N,k} \in \mathcal{F}(h,R|I_k)$, it follows that $b(N) \cdot f_{N} \leq \log^+(1/w)$ on $\T$, and so we ensure $(i)$ if $b(N)$ is chosen sufficiently large. In particular, this holds if a priori we have $b(N) \to \infty$ as $N \to \infty$, which we will be able to ensure.

If $I$ is any interval on $\T$, then $b(N)\cdot f_N$ integrates to zero over any interval $I_k$ contained in $I$, and so if $I_l$ and $I_r$ are the the two intervals among $\{I_k\}_{k=1}^N$ which contain the endpoints of $I$, we obtain
\begin{align*}
\int_I b(N) \cdot f_N \, d\m &= \int_{I_l} b(N) \cdot f_N \, d\m + \int_{I_r} b(N) \cdot f_N \, d\m \\
&\leq  \int_{I_l \cap I} g_{N,l} \, d\m + \int_{I_r \cap I} g_{N,r} \, d\m \\
&\leq 2\cdot h(|I|).
\end{align*}   We used that $g_{N, l} \in \mathcal{F}(h,R|I_l)$ and $g_{N,r} \in \mathcal{F}(h,R|I_r)$ in the last inequality. Again, we obtain $(ii)$ if $b(N) \to \infty$ as $N \to \infty$.

Part $(iii)$ holds since \[b(N) \cdot \int_\T f_N \, d\m = b(N) \cdot \sum_{k: |I:k \cap A_N| > 0} \int_{I_k} f_N \, d\m = 0.\]

It remains to show that choices of $M(N)$ in \eqref{E:ANdef} and $b(N)$ in \eqref{E:f_Ndef} can be made so that $(iv)$ holds. Since $A_N \subset \resh{h}{w}$ implies $|I_k \cap A_N| \leq |I_k \cap \resh{h}{w}|$, and $g_{N,k} \leq R \leq M(N)$ on $A_N$, we may use \eqref{E:GNkLowerMassEst} and \eqref{E:CNkdef} to estimate 
\begin{align*}
    |I_k \cap \resh{h}{w}| \cdot C_{N,k} &\geq |I_k \cap A_N| \cdot C_{N,k} \\
    &= \int_{I_k} g_{N,k}\, d\m - \int_{I_k \cap A_N} g_{N,k} \, d\m \\
    &\geq \frac{h(|I_k \cap \resh{h}{w}|)}{6} - |I_K \cap A_N| \cdot M(N) \\
    &\geq \frac{h(|I_k \cap \resh{h}{w}|)}{6} - |I_k \cap \resh{h}{w}| \cdot M(N).
\end{align*} If $|I_k \cap A_N| > 0$, then we may divide by the larger quantity $|I_k \cap \resh{h}{w}|$ to obtain
\begin{equation}
    \label{E:CNkLowerEst}
    C_{N,k} \geq \frac{h(|I_k \cap \resh{h}{w}|)}{6 |I_k \cap \resh{h}{w}|} - M(N) \geq \frac{h(1/N)}{6/N} - M(N). 
\end{equation} The second of the above inequalities holds since $|I_k \cap \resh{h}{w}| \leq |I_k| = 1/N$, and $h$ satisfies property (R1) stated in Section \ref{S:OptimProbSubsec}. Set \[ M(N) = b(N) = \sqrt{N\cdot h(1/N)}.\] Then $M(N), b(N)$ tend to $\infty$ as $N \to \infty$ as a consequence of (R2), and we obtain from \eqref{E:CNkLowerEst} that \[ \frac{C_{N,k}}{b(N)} \geq \frac{\sqrt{N h(1/N)}}{12}\] for $N$ sufficiently large. Since $f_N = -C_{N,k}/b(N)$ on $I_k \cap A_N$ irrespective of $k$, we obtain \[ f_N \leq -\frac{\sqrt{N h(1/N)}}{12}\] on $A_N$. Since the sets $A_N$ are increasing with $N$, and $\cup_N A_N$ is a set of full measure in $\resh{h}{w}$, property (R2) ensures that $(iv)$ also holds. 
\end{proof}

It is well known that the outer function $F_N$ satisfies $|F_N| = |\exp(f_N)|$ $\m$-almost everywhere on $\T$ . Thus for any small $\epsilon$ and all sufficiently large $N$, we obtain from $(i)$ of \thref{L:ConstructionLemma} that
\[ |F_N|^t w \leq \exp \big(t \epsilon \log^+(1/w)\big)w \leq w + w^{1- t\epsilon} \in \mathcal{L}^1(d\m),\] which shows that the family $\{F_N\}_N$ is uniformly norm bounded in $\mathcal{L}^{t}(w\, d\m)$ for any $t \in (0, \infty)$. We have also the estimate
\begin{equation}
    \label{E:fNDiskEstimate}
    |F_N(z)| \leq \exp\Big( \epsilon \frac{h(1-|z|)}{1-|z|)} \Big), \quad z \in \D 
\end{equation} for any small $\epsilon > 0$ and all sufficiently large $N$. This follows immediately from a classical Poisson integral estimate and part $(ii)$ of \thref{L:ConstructionLemma}. Namely, if $h$ is a gauge function, $\sigma$ is a real-valued Borel measure satisfying $\sigma(\T) = 0 $, and $\sigma(I) \leq h(|I|)$ for all intervals $I \subset \T$, then there exists a constant $C = C(h) > 0$ for which the estimate
\begin{equation}
    \label{E:PoissonEstimate0}
    \Re \Bigg(  \int_{\T} \frac{\zeta + z}{\zeta - z} d\sigma(\zeta) \Bigg) = 
    \int_{\T} \frac{1-|z|^2}{|\zeta-z|^2} d\sigma(\zeta) \leq C \cdot\frac{h(1-|z|)}{1-|z|}, \quad z \in \D
\end{equation} holds. See \cite[p. 297]{havinbook} for a simple proof.

For $h$ as in \eqref{E:CarlesonGaugeDef}, whenever $\epsilon$ is so small that $\alpha - t\epsilon > -1$ and $N$ is sufficiently large, we obtain from \eqref{E:fNDiskEstimate} that 
\[ |F_N(z)|^t (1-|z|)^{\alpha} \leq 3 \cdot (1-|z|)^{\alpha-t\epsilon} \in \mathcal{L}^1(dA).\] Thus $\{F_N\}_N$ is a uniformly norm bounded subset of $\mathcal{L}^t(dA_\alpha)$ for any $t \in (0,\infty)$. Similarly, we obtain from \eqref{E:fNDiskEstimate} that for any gauge function $h$ and all large $N$ we have \[|F_N(z)|^t G_{ah}(z) \leq 1, \quad z \in \D.\]
Combining our estimates, we have just verified that for any of the considered measures $\mu$ and any $t > 0$, the family $\{F_N\}_N$ is uniformly norm bounded in $\mathcal{L}^t(\mu)$. 

\begin{proof}[Proof of the direct sum decomposition in \thref{T:ThomsonDecompTheorem} and in \eqref{E:GeneralizedThomsonDecomp}]
Let $t^* > \max(t,1)$. By reflexivity of $\mathcal{L}^{t^*}(\mu)$, we may pass to a subsequence to ensure that the sequence $\{F_N\}_N$ constructed above converges weakly to some $F \in \mathcal{L}^{t^*}(\mu)$. Since $\Po^{t^*}(\mu)$ is norm-closed in $\mathcal{L}^{t^*}(\mu)$, it is also weakly closed, and so $F \in \Po^{t^*}(\mu)$. The estimate \eqref{E:fNDiskEstimate} shows that $\{F_N\}_N$ is a normal family of analytic functions in $\D$, so by refining the subsequence we may assume that $\{F_N\}_N$ converges pointwise in $\D$.  Note that from $t^* > t$ and the containment $F \in \Po^{t^*}(\mu)$ we obtain that $F \in \Po^t(\mu)$, since a sequence of analytic polynomials converging to $F$ in $\Po^{t^*}(\mu)$ converges also to $F$ in $\Po^t(\mu)$. We need to verify that $F$ satisfies the conditions of \thref{L:ThompsonDecompReductionLemma}. By property $(iii)$ of \thref{L:ConstructionLemma}, we get $F(0) = \lim_N F_N(0) = 1$. Also, by \eqref{E:fNDiskEstimate} and by letting $\epsilon \to 0$, we get $|F(z)| = \limsup_{N \to \infty} |F_N(z)| \leq 1$ for $z \in \D$. By the maximum modulus principle, we obtain that $F \equiv 1$ in $\D$. We see  from part $(iv)$ of \thref{L:ConstructionLemma} that $|F_N| \leq 1$ on $\resh{h}{w}$ and also $|F_N| \to 0$ $d\m$-almost everywhere on $\resh{h}{w}$. The weak convergence in $\mathcal{L}^{t^*}(\mu)$ implies in particular that the restrictions $F_N|\T$ converge weakly in $\mathcal{L}^{t^*}(w \, dm)$, and the weak $\mathcal{L}^{t^*}(w\, d\m)$ limit equals $F|\T$. We readily deduce from the pointwise convergence of $\{F_N\}_N$ to $0$ on $\resh{h}{w}$ that $F \equiv 0$ on $\resh{h}{w}$. Thus $F$ satisfies the hypotheses of \thref{L:ThompsonDecompReductionLemma}, and our proof is complete.
\end{proof}

\subsection{Irreducibility proof} 
\label{S:IrreducibilityProofSec}

The irreducibility of the first summand on the right hand side in the displayed equation in \thref{T:ThomsonDecompTheorem} holds if we assume some additional properties of $h$. One of them is
\begin{enumerate}[(R3)]
    \item For some constant $C = C(h)$ we have $\int_{0}^a h(x) \, dx \leq C \cdot h(a)$, $a \in (0,1)$.
\end{enumerate}
An easy computation shows that $h(x) = x \log(e/x)$ satisfies (R3).
The proof of irreducibility under this assumption is essentially contained in Khrushchev's article \cite{khrushchev1978problem} and is stated without proof in \cite{kriete1990mean}. We need a simple adaptation. 

\begin{lem}[\textbf{Khrushchev's Uniqueness Theorem}]
    \thlabel{L:KhrushchevLemma} Assume that the gauge function $h$ satisfies the additional condition (R3). Let $E \in \assoch{h}{w}$ have positive Lebesgue measure and $w_E$ be the restriction of $w$ to $E$. If $\{p_n\}_n$ is a sequence of analytic polynomials, $p_n \to f$ in $\mathcal{L}^t(w_E\, d\m)$, and if the following two conditions hold:
    \begin{enumerate}[(i)]
        \item $|p_n(z)| \leq \exp\Big( a \frac{h(1-|z|)}{1-|z|)} \Big)$ for some $a > 0$ and for all $z \in \D$,
        \item $p_n(z) \to 0$ for $z \in \D$,
    \end{enumerate}
    then $f \equiv 0$.
\end{lem}

\begin{proof}
The only technical part of the proof is the estimation of harmonic measure of a certain domain, which Khrushchev accomplishes in \cite[Section 3]{khrushchev1978problem}. He defines the domain $\Omega = \D \setminus \big(\bigcup_n B(\ell_n) \big)$, where $B(\ell_n) = \{ z \in \D : z/|z| \in \ell_n, |z| > 1-|\ell_n| \}$ is a curvilinear box, and $\ell_n$ is a maximal open interval complementary to $E$ in $\T$. It is not hard to verify that $\partial \Omega$ is a rectifiable Jordan curve. 

We denote by $\omega$ the harmonic measure in $\Omega$ at $0 \in \Omega$. In \cite[Proof of Theorem 3.1]{khrushchev1978problem}, Khrushchev shows that under condition (R3) on $h$, we have
\begin{equation}
    \label{E:KhrushchevEstimate}
    \int_{\partial \Omega \cap \D} \frac{h(1-|z|)}{1-|z|} d\omega(z) < \infty. 
\end{equation}

We use this result as a given, and proceed with the rest of the proof. Introduce the outer function $W_E$ in $\D$ with the following modulus $|W_E|$ on $\T$:
\begin{align}
    \label{E:WEMultDef}
    |W_E(\zeta)| = \begin{cases} 
    1, & \zeta \in \T \setminus E, \\ 
    \min \big(1 , (w(\zeta))^{1/t}\big), & \zeta \in E.    
    \end{cases}
\end{align}
Such a function $W_E$ exists as a consequence of $E \in \assoc{w}$, since it implies that for $|W_E|$ in \eqref{E:WEMultDef} we have $\log |W_E| \in \mathcal{L}^1(d\m)$. Note that $|W_E(z)| \leq 1$ for $z \in \D$, since $|W_E| \leq 1$ on $\T$.

Let $z \in \partial \Omega \cap \D$. Then by $(i)$ we have
\[ \log^+ (|W_E(z) \cdot p_n(z)|) \leq a \frac{h(1-|z|)}{1-|z|}, \quad z \in \D.\] By convergence of $\{p_n\}_n$ in $\mathcal{L}^t(w_E \, d\m)$, and by the inequality $|W_E|^t \leq w$ which holds $\m$-almost everywhere on $E$, we also have
\begin{equation}
    \label{E:PnTEstimate}
    \sup_n \int_E |W_E \cdot p_n|^t \,d\m \leq \sup_n \int_E |p_n|^t w \, d\m < \infty.
\end{equation} Khrushchev in his proof observes that $d\omega$ restricted to $E$ is absolutely continuous with respect to $d\m$, and that $d\omega \leq d\m$ on $E$. This follows by monotonicity of the harmonic measure (see \cite[Corollary 4.3.9]{ransford1995potential}) applied to the inclusion $\Omega \subset \D$. This observation, together with hypothesis $(i)$, and the estimates \eqref{E:KhrushchevEstimate} and \eqref{E:PnTEstimate}, implies that 
\[ \sup_n \int_{\partial \Omega} \log^+ (|W_E \cdot p_n|) \, d\omega < \infty. \] If $\phi: \D \to \Omega$ is a conformal map fixing the origin, then a change of variables and the conformal invariance of harmonic measure, namely $d\m = d(\omega \circ \phi)$, shows that 
\[ \sup_n \int_\T \log^+( |W_E \circ \phi) \cdot (p_n \circ \phi) |) \, d\m < \infty.\] 
The rectifiability of $\partial \Omega$ implies that $\phi$ is conformal at $\m$-almost every point on $\T$ (see \cite[VI.1]{garnett}) which implies that $|(W_E \circ \phi) \cdot (p_n \circ \phi) |$ coincides $\m$-almost everywhere with the non-tangential boundary values of the bounded functions $(W_E \circ \phi) \cdot (p_n \circ \phi)$ defined in $\D$. Hence the finiteness of the supremum above shows that $\{(W_E \circ \phi) \cdot (p_n \circ \phi)\}_n$ is a bounded subset of the Nevanlinna class in $\D$. Since after passing to a subsequence we may assume that $p_n \to f$ almost everywhere on $E$, we may also assume that $(W_E \circ \phi) \cdot (p_n \circ \phi) \to (W_E \circ \phi) \cdot (f \circ \phi)$ almost everywhere on $\widetilde{E} = \phi^{-1}(E)$. We have $|\widetilde{E}| > 0$ since $|\widetilde{E}| = \omega(E) > 0$ as a consequence of $\omega$ and arclength measure on $\partial \Omega$ being mutually absolutely continous (see \cite[VI.1.2]{garnett2005harmonic}). We are now in the setting of the classical Khinchin-Ostrowski theorem \cite[Part II, 2.3]{havinbook}). This theorem implies that $(W_E \circ \phi) \cdot (f \circ \phi) \equiv 0$ as a consequence of $(W_E \circ \phi) \cdot (p_n \circ \phi) \to 0$ pointwise on $\D$, which follows by our hypothesis $(ii)$. Hence $f \equiv 0$, since $|W_E| > 0$ $\m$-almost everywhere on $E$.
\end{proof}

To apply \thref{L:KhrushchevLemma} in the setting of $\Po^t(\mu)$ spaces, we need a bound of the form
\begin{equation} \label{E:pGrowthBound} |p(z)| \leq C \exp\Big( a \frac{h(1-|z|)}{1-|z|)} \Big) \end{equation} where $p$ is a polynomial and $C > 0$ is a constant depending only on the norm of $p$ in $\Po^t(\mu)$. If the bound exists then any sequence $\{p_n\}_n$ of polynomials in $\Po^t(\mu)$ converging to an element $f \in \Po^t(\mu)$ with $f|\D \equiv 0$ satisfies the assumptions of \thref{L:KhrushchevLemma} for every $E \in \assoch{h}{w}$, $|E| > 0$. The conclusion is that $f|E \equiv 0$, and since a countable union of sets in $\assoch{h}{w}$ constitutes a carrier for $w_c$, we deduce that $f|\T \equiv 0$ as an element of $\mathcal{L}^t(w_c \,  d\m)$. Thus $f \equiv 0$ in $\Po^t(\mu)$ if $w$ is residual-free, in the sense that $w = w_c$. This establishes irreducibility of $\Po^t(G_{ah}dA + w_c\, d\m)$ for any $h$ admitting a bound of the form \eqref{E:pGrowthBound}. 

In the case that $h$ has the form \eqref{E:CarlesonGaugeDef}, the bound \eqref{E:pGrowthBound} reduces to \[|p(z)| \leq C(1-|z|)^{-a}\] for some $a > 0$. A well-known estimate in Bergman spaces states that
\[ |p(z)| \lesssim  \Big(\int_{\D} |p_n|^t dA_\alpha \Big)^{1/t} \cdot (1-|z|)^{-a}, \quad z \in \D\] for some $a = a(\alpha, t) > 0$ (see, for instance, \cite[Theorem 1 of Chapter 3]{durenbergmanspaces}). Thus for $\mu$ of the form \eqref{E:MuDef}, a convergent sequence of polynomials $\{p_n\}_n$ in $\Po^t(\mu)$ satisfies assumption $(i)$ in \thref{L:KhrushchevLemma}, and our discussion in the previous paragraph gives us the irreducibility of the piece $\Po^t(dA_\alpha + w_c\, d\m)$ in the decomposition in \thref{T:ThomsonDecompTheorem}. This finishes the proof of that result.

Although we shall not explicitly verify the conditions, we mention that the regularity conditions (R1), (R2) and (R3), and the estimate \eqref{E:pGrowthBound}, hold for the particular choice $h(x) = x^{\beta}$ for any $\beta \in (0,1)$. As a consequence, the following variant of \thref{T:ThomsonDecompTheorem} holds.

\begin{thm*}
    Let $h(x) = x^\beta$ for some $\beta  \in (0, 1)$, and let $\mu$ have the form \eqref{E:MuGeneralizedDef} for some $a > 0$.
    For any $t \in (0,\infty)$, we have
\[
    \Po^t(\mu) = \Po^t( G_{a h}dA + w_c \, d\m) \oplus \mathcal{L}^t(w_r \, d\m) \] where  \[ w_c = w|\coreh{h}{w}, \quad w_r = w|\resh{h}{w}\] and $\Po^t( G_{a h}dA + w_c \, d\m)$ is irreducible.
\end{thm*}

\section{Permanence of singular measures living on associated sets}

\label{S:PermanenceSubsec}

\subsection{The goal} The result to be proved in this section applies strictly to measures $\mu$ of the form \eqref{E:MuDef}, and not to the class \eqref{E:MuGeneralizedDef}. We shall not go into further details regarding the generality in which the argument given below applies, but we wish to mention that the choice $h(x) = x^\beta$, $\beta \in (0,1)$, leads to some issues in the proof.

\begin{prop} [\textbf{Permanence of singular inner factors}]\thlabel{P:PermanenceInnerAssocSet}
Let $\mu$ be of the form \eqref{E:MuDef}. If $t ´\in (0, \infty)$, $\Po^t(\mu)$ is a space of analytic functions and $S_{\nu}$ is a singular inner function for which the corresponding singular measure $\nu$ is supported on a set $E \in \assoc{w}$, then any function $h \in [S_\nu] \cap \hil^\infty$ satisfies $h/S_\nu \in \hil^\infty$. 
\end{prop}

The result is a type of \textit{indestructibility} property (the term was coined in \cite{hedenmalmbergmanspaces}) of $S_\nu$ under convergence in $\Po^t(\mu)$ norms. In relation to \thref{T:CyclicityMainTheorem}, it shows the necessity of vanishing of $\nu$ on associated sets for cyclicity of $S_\nu$. 

As remarked in the Introduction, the case $t=2$ of \thref{P:PermanenceInnerAssocSet} has been established in \cite{limani2024constructions}. The proof given there essentially involves a functional analytic argument. The proof of the proposition in the general case, which we will present here, involves a simple extension to associated sets of an argument in \cite{berman1984cyclic} by Berman, Brown and Cohn which applies to $\BC$-sets of Lebesgue measure zero. Before explaining their work, let us state two lemmas which we will use at a later stage in the proof. The first of them is the permanence property in $\hil^t$.

\begin{lem}
\thlabel{L:HpPerm}
Assume that $\{f_n\}_n$ is a sequence of functions in $\hil^t = \Po^t(d\m)$, satisfying \[ \sup_n \int_\T |f_n|^t \, d\m < \infty.\] Let $f \in \hil^t$, and assume that $S_\nu f_n \to f$ pointwise in $\D$. Then $f/S_\nu \in \hil^t$.
\end{lem}

The claim is certainly true for $t=2$. In that case, we may extract a weakly convergent subsequence of $\{S_\nu f_n\}_n$, which must converge to a function in the closed subspace $\S_{\nu} \hil^2$. Then $f \in S_\nu\hil^2$, so that $f/S_\nu \in \hil^2$. For $t \neq 2$, we may reduce to the previous case by multiplying the sequence $\{f_n\}_n$ by appropriate outer factors. For details, we refer to \cite[proof of Theorem 1.4]{limani2021abstract}.

\begin{lem} \thlabel{L:OuterComp} Let $F$ be a bounded outer function. If $\phi:\D \to \D$ is analytic, then $F \circ \phi$ is outer.
\end{lem}

\begin{proof}
    A bounded analytic function $F$ in $\D$ is outer if and only if a sequence $\{h_n\}_n$ of functions $h_n \in \hil^\infty$ exists for which we have the convergence $h_n  F\to 1$ pointwise in $\D$ and a uniform bound $\sup_{z \in \D} |F(z)h_n(z)| < C$ independent of $n$. If $\phi$ is as in the lemma, then for $h^*_n := h_n \circ \phi \in \hil^\infty$ we have that $h_n^* \cdot (F \circ \phi)$ converges pointwise to $1$ in $\D$  and satisfies the mentioned bound. So $F \circ \phi$ is outer.
\end{proof}

\subsection{A subdomain of the disk} \label{S:SubdomainSubsec} In \cite{berman1984cyclic}, the authors use a subdomain $\Omega$ of $\D$ with smooth boundary $\partial \Omega$ which satisfies $\partial \Omega \cap \T = E$. If $\{\ell_n\}_n$ is the sequence of maximal open intervals complementary to $E$ on $\T$, and $a_n, b_n$ are the endpoints of the interval $\ell_n$, then they define the curve $\gamma \subset \D \cup \T$ by
\[ \gamma = \big\{ (1-r(t))e^{it} : t \in [0, 2\pi) \big\}\] where $r(t)$ is chosen so that $r(t) \simeq c (\dist{e^{it}}{E})^2$. The constant $c > 0$ is soon to be specified. More precisely, we will set $r(t) = 0$ if $e^{it} \in E$ and otherwise
\begin{equation*}
    r(t) = c\frac{|a_n-e^{it}|^2|b_n - e^{it}|^2}{|\ell_n|^2}, \quad e^{it} \in \ell_n.
\end{equation*}

The domain $\Omega$ is defined to be the interior of the curve $\gamma$. The choice of the function $r$ makes $\partial \Omega = \gamma$ smooth enough to apply Kellogg's theorem, \cite[Theorem II.4.3]{garnett2005harmonic}, which asserts that any conformal mapping $\phi: \D \to \Omega$ has a derivative $\phi'$ which extends to a continuous non-vanishing function on $\T$. We will fix $\phi$ which satisfies $\phi(0) = 0$.  

An important property of $\Omega$ is that $S_\nu$ is bounded from below on $\D \setminus \Omega$. To see this, note first that since $\nu$ is supported on $E$, for any $z \in \D \setminus \Omega$ and $z^* := z/|z| \in \T \setminus E$, we have that $S_\nu$ is holomorphic in a neighbourhood of $z^*$, and $|S_\nu(z^*)| = 1$. Secondly, the explicit formula \eqref{E:SnuEq} and the chain rule imply that for $z \in \D$ we have
\begin{align*}
|S'_\nu(z)| &= \Big| \int_E \frac{2\zeta}{\zeta - z} d\nu(\zeta) \Big| |S_\nu(z)| \\
&\leq 2\nu(\T)\cdot \dist{z}{E}^{-2}.
\end{align*} 
Thirdly, if $z \not\in \Omega$, then the distance between $z$ and $z^*$ is dominated by a constant multiple of $c ´(\dist{z^*}{E})^2$. Finally, consider any $w \in \D \setminus \Omega$ and let $L$ be the straight line segment between $w$ and $w^* = w/|w| \in \T$. For any $z \in L$, we have also that $\dist{z}{E}$ dominates a constant multiple of $\dist{z^*}{E} = \dist{w^*}{E}$. Using in combination our above observations, we obtain that
\begin{align*}
    1 - |S_\nu(w)| &\leq |S_\nu(w^*) - S_\nu(w)| \\ &= \Big| \int_L S'_\nu(z) \, dz \Big| \\ &\lesssim \int_L (\dist{z}{E})^{-2} |dz| \\
&\lesssim |w - w^*| (\dist{w^*}{E})^{-2} \\ &\lesssim c.
\end{align*} It follows that we may choose $c > 0$ small enough to ensure that \begin{equation}
    \label{E:SnuLowerEst}
    |S_\nu(z)| > 1/2, \quad z \in \D \setminus \Omega.
\end{equation}

\subsection{Singular factor of $S_\nu \circ \phi$} The non-vanishing condition on $\phi'|\T$ allows the authors in \cite{berman1984cyclic} to conclude that the function $S_\nu \circ \phi$ has a non-trivial singular inner factor.

\begin{lem} \thlabel{L:SingCompInnerFactor} The composed function $S_\nu \circ \phi : \D \to \D$ has the inner-outer factorization\[ S_\nu \circ \phi = \theta \cdot U,\] where $\theta$ is a non-trivial singular inner function, and $U$ is an outer function bounded from above and below in $\D$.
\end{lem}

\begin{proof}
The non-triviality of $\theta$ is a consequence of \cite[Theorem 2.1]{berman1984cyclic}. We clearly have $|U| = |S_\nu \circ \phi| \leq 1$ in $\D$, so it remains to verify that $U$ is bounded from below in $\D$. Since $U$ is outer, this we can do by showing that $|U|$ is essentially bounded from below on $\T$. Certainly we have $|U| = 1$ $\m$-almost everywhere on $\phi^{-1}(E)$, for $\phi$ respects subsets of full $\m$-measure in $E$ and $\phi^{-1}(E)$, and so for $\m$-almost every $w = \phi(z) \in E$ and $\m$-almost every $z \in \phi^{-1}(E)$ we have \[ 1 = |S_\nu(w)| = |\theta(z) \cdot U(z)| = |U(z)|.\]  If $z \in \T \setminus \phi^{-1}(E)$, then $w = \phi(z) \in \D \cap \partial \Omega$, and so $|S_\nu(w)| > 1/2$ according to Section \ref{S:SubdomainSubsec}. Therefore
\[ |U(z)| \geq |\theta(z) \cdot U(z)| = |S_\nu(w)| > 1/2.\]
\end{proof}

\subsection{Proof of the permanence property}

In the notation of \thref{P:PermanenceInnerAssocSet}, the function $h/S_\nu$ is bounded outside of $\Omega$ as a consequence of \eqref{E:SnuLowerEst}. What remains to be shown is that it is bounded inside of $\Omega$ also. Because $h \in [S_\nu]$, there exists a sequence $\{p_n\}_n$ of polynomials for which $S_\nu \cdot p_n \to h$ in the norm of $\Po^t(\mu)$. In particular, we have convergence pointwise for every $z \in \D$. If the composed sequence $(S_\nu \circ \phi)\cdot (p_n \circ \phi)$ were bounded in the norm of $\hil^t$, then in the notation of \thref{L:SingCompInnerFactor} we could conclude that $(h \circ \phi)/\theta \in \hil^\infty$, and consequently $h/S_\nu$ would be bounded in $\Omega$. We can't quite obtain the desired norm bound for $(S_\nu \circ \phi)\cdot (p_n \circ \phi)$, but we can after a multiplication of the sequence with appropriate outer functions.

We introduce the outer function $F_E$  satisfying the following equality $\m$-almost everywhere on $\T$:

\begin{align}
    \label{E:FEMultDef}
    |F_E(\zeta)| = \begin{cases} 
    \dist{\zeta}{E}, & \zeta \in \T \setminus E, \\ 
    1, & \zeta \in E.    
    \end{cases}
\end{align}
That is,
\begin{equation*}
\label{E:FEMultFormula}
F_E(z) = \exp \Big( \int_{\T \setminus E} \frac{\zeta + z}{\zeta - z} \log \dist{\zeta}{E} \, d\m(\zeta) \Big), \quad z \in \D.  
\end{equation*}

The convergence of the integral in the definition of $F_E$ is assured as a consequence of $\log \dist{\zeta}{E} \in \mathcal{L}^1(d\m)$ for sets $E \in \BC$ and  a short computation showing that \[\int_\ell \log \dist{\zeta}{E} d\m(\zeta) \simeq |\ell| \log (e/|\ell|)\] for short enough maximal intervals $\ell$ complementary to $\T$. Note that by adding finitely many points to $E$, we may ensure that the complementary intervals $\ell$ have a length bounded from above by any desired small constant. The operation of adding finitely many points to $E$ clearly does not affect its membership in $\assoch{h}{w}$.

If $w \in \partial \Omega \cap \D$, then $w^* = w/|w| \in \T \setminus E$, and by construction of $\Omega$ we have, for some constant $c > 0$, \begin{equation}
\label{E:wwstarDistEst} \dist{w}{w^*} = 1-|w| \simeq (\dist{w^*}{E})^2.
\end{equation} The harmonic measure in $\D$ at $w$ of the interval $I_w$ of length $(1-|w|)/2$ centered at $w^*$, namely
\[ \int_{I_w} \frac{1-|w|^2}{|\zeta - w|^2} d\m(\zeta),\] is then readily seen to be bounded from below independently of $w$, say by $\delta > 0$. If the complementary interval $\ell$ containing $w^*$ is sufficiently short, then by \eqref{E:wwstarDistEst} we have \[ \frac{1-|w|}{2} \leq \dist{w^*}{E}.\] Thus, by the remark in previous paragraph, we may assume that $I_w$ does not intersect $E$. By this observation and \eqref{E:wwstarDistEst}, we conclude that for every $\zeta \in I_w$, we have the estimate
\[F_E(\zeta) = \dist{\zeta}{E} \leq \dist{w^*}{E} + \frac{1-|w|}{2} \lesssim \sqrt{1-|w|}.\] So, since $|F_E|$ is bounded in $\D$, we obtain by the classical Two Constants Theorem that 
\begin{equation}
    \label{E:FEDecayEstimate}
    |F_E(w)| \lesssim \Big( \sup_{\zeta \in I_w} |F_E(\zeta)| \Big)^\delta \lesssim (1-|w|)^{\delta/2}, \quad w \in \partial \Omega \cap \D.
\end{equation}
By replacing $F_E$ by a positive power of itself, we may replace $\delta/2$ above by a positive number arbitrarily large. 

\begin{proof}[Proof of \thref{P:PermanenceInnerAssocSet}]
Assume that $h \in [S_\nu] \cap \hil^\infty$, and that $S_\nu \cdot p_n \to h$ in $\Po^t(\mu)$, $\{p_n\}_n$ being a sequence of polynomials. Since $\mu$ has form \eqref{E:MuDef}, the convergence implies that we have the two bounds
\[ \sup_{n} \int_\D |S_\nu \cdot p_n|^t dA_\alpha < \infty\]
and 
\begin{equation}
\label{E:CircleBoundIneq}
\sup_{n} \int_\T |S_\nu \cdot p_n|^t w \, d\m < \infty.
\end{equation}  

The first of the two bounds implies the growth estimate 
\[ |S_\nu(z) \cdot p_n(z)| \leq \frac{C}{(1-|z|)^B}, \quad z \in \D\] with constants $B, C > 0$ independent of $n$ (see \cite[Theorem 1 of Chapter 3]{durenbergmanspaces}). As a consequence of \eqref{E:FEDecayEstimate}, or more precisely as a consequence of the comment immediately following that equation, we may assume that
\[ \sup_{z \in \partial \Omega \cap \D} |F_E(z) \cdot S_\nu(z) \cdot p_n(z)| \leq 1.\]

Recall the definition of the outer function $W_E$ in \eqref{E:WEMultDef}. If $|E| = 0$ then $W_E$ reduces to the constant $1$, and in that case $W_E$ plays no role in the following estimates. In any case, as a consequence of the inequality $|W_E| \leq w^{1/t}$ which holds $\m$-almost everywhere on $E$, the second bound \eqref{E:CircleBoundIneq} tells us that \[ \sup_{E} \int_E |W_E \cdot S_\nu \cdot p_n|^t \, d\m < \infty.\] The bound holds trivially if $|E| = 0$, of course. Since $\partial \Omega = E \cup (\partial \Omega \cap \D)$, the two estimates imply
\[ \sup_n \int_{\partial \Omega} |F_E \cdot W_E \cdot S_\nu \cdot p_n|^t ds < \infty,\] where $ds$ denotes the arclength element on $\partial \Omega$. Setting \[g_n := F_E \cdot W_E \cdot S_\nu \cdot p_n \in \hil^\infty\] and changing variables, we obtain
\[ \sup_n \int_\T |g_n \circ \phi|^t |\phi'| d\m < \infty\] where $\phi: \D \to \Omega$ is the earlier introduced conformal map. Since $|\phi'|$ is bounded from below on $\T$, we obtain that $\{g_n \circ \phi\}_n$ is a bounded subset of the Hardy space $\hil^t$. The pointwise convergence $S_\nu \cdot p_n \to h$ in $\D$ implies that $g_n \circ \phi \to (F_E \cdot W_E \cdot h) \circ \phi$ in $\D$, and since $g_n \circ \phi$ has $\theta$ as an inner factor according to \thref{L:SingCompInnerFactor} (more precisely, we mean that the quotient of $g_n \circ \phi$ and $\theta$ is bounded in $\D$), so does $(F_E \cdot W_E \cdot h) \circ \phi$, this time by \thref{L:HpPerm}. The functions $F_E \circ \phi$ and $W_E \circ \phi$ are outer by \thref{L:OuterComp}. Hence actually $h \circ \phi$ has the inner factor $\theta$. In notation of \thref{L:SingCompInnerFactor}, the function $\dfrac{h \circ \phi}{\theta} = \dfrac{(h \circ \phi) \cdot U }{ \theta \cdot U}$ is bounded in $\D$, which is equivalent to $\dfrac{h \cdot (U \circ \phi^{-1})}{S_\nu}$ being bounded in $\Omega$. Since $U$ is bounded from below, $h/S_\nu$ is bounded in $\Omega$. By an earlier remark, $h/S_\nu$ is bounded in $\D \setminus \Omega$, and so $h/S_\nu \in \hil^\infty$. The proof is complete.
\end{proof}

\section{Cyclicity of singular inner functions} \label{S:CyclicitySection}

\subsection{The goal}
The result to be proved in this section is the sufficiency part of \thref{T:CyclicityMainTheorem}, and also the following variant.

\begin{prop}[\textbf{Main Cyclicity Result}] \thlabel{P:CyclicityProp} Let $\mu$ be of a measure of the form \eqref{E:MuGeneralizedDef}. If $t \in (0, \infty)$, $\Po^t(\mu)$ is a space of analytic functions, and $S_{\nu}$ is a singular inner function for which the corresponding singular measure satisfies
\begin{equation}
    \label{E:NuVanishAssoc1}
\nu(E) = 0, \quad E \subset \assoch{h}{w},
\end{equation} then $S_{\nu}$ is cyclic in $\Po^t(\mu)$.
\end{prop}

We make a few remarks before proceeding with the proof.

\begin{itemize}
    \item It suffices to prove the result for large $t > 1$. Indeed, if $t_1 < t_2$ and $S_{\nu}$ is cyclic in $\Po^{t_2}(\mu)$, then denoting by $\| \cdot \|_{\mu,t}$ the norm in $\mathcal{L}^t(\mu)$, there exists a sequence $\{p_n\}_n$ of polynomials such that $\|S_{\nu}p_n - 1\|_{\mu, t_2} \to 0$. By Jensen's inequality applied to the convex function $x \mapsto x^{t_2/t_1}$ and a normalized version of the measure $\mu$, we obtain readily that $\|S_{\nu}p_n - 1\|_{\mu, t_1} \to 0$. In particular, we may in our proofs work in the more convenient reflexive range $t \in (1,\infty)$. It is worth noting here that increasing the parameter $t$ does not affect analyticity of the space $\Po^t(\mu)$. To see this, note that if $\Po^{t_2}(\mu)$ is not analytic, and so contains a function $f$ satisfying $f|\D \equiv 0$, then whenever $t_1 < t_2$, we have also that $f \in \Po^{t_1}(\mu)$, and consequently $\Po^{t_1}(\mu)$ is not analytic either. In fact, for most natural gauge functions $h$, such as the one in \eqref{E:CarlesonGaugeDef}, it is a consequence of \thref{T:ThomsonDecompTheorem}, or the variants presented in Section~\ref{S:IrreducibilityProofSec}, that analyticity of $\Po^t(\mu)$ is independent of the choice of the parameter $t$, since it is decided entirely by the structure of the weight $w$ appearing in \eqref{E:MuDef} or \eqref{E:MuGeneralizedDef}. 
    \item It suffices to show that $S_{\nu}^{1/B}$ is cyclic in $\Po^t(\mu)$ for some positive integer $B$. Indeed, if $m_1$ and $m_2$ are two cyclic \textit{bounded} functions, then so is their product. The claim follows from the estimate 
    \[ \|m_1m_2pq - 1\|_{\mu,t} \leq \|m_2q - 1\|_{\mu,t} \cdot \|m_1p\|_\infty + \|m_1p - 1\|_{\mu,t}\] where $p$ and $q$ are polynomials. If $m_1$ is cyclic, we may choose $p$ to make the second term on the right-hand side small. Having fixed $p$, if $m_2$ s cyclic we may choose $q$ to make the first term on the right-hand side small. Cyclicity of $m_1m_2$ follows.
\item We emphasize again that the necessity for cyclicity of $S_\nu$ of vanishing of $\nu$ on associated set has been established in Section \ref{S:PermanenceSubsec} only for the case of measures $\mu$ of the form \eqref{E:MuDef}. In particular, the converse has not been established in the context of more general gauges $h$. Thus we make no claims regarding the necessity of the vanishing condition in \thref{P:CyclicityProp}.
\end{itemize}

\subsection{Modification of Korenblum's linear program}
\label{S:KorenblumProgramSubsec}

Let $\nu$ be a finite non-negative singular Borel measure on $\T$, and set 
\begin{equation}
    \label{E:nu0def}
    \nu_0 =: \nu - \nu(\T) d\m.
\end{equation}
Let $N$ be a positive integer, and $\{I_s\}_{s=1}^N$ be a partition of $\T$ into the half-open intervals \[ I_s := \{ e^{it} : 2\pi(s-1)/N \leq t < 2\pi s/N \}, \quad s = 1, \ldots, N.\] Note that the partition depends on the choice of $N$, but for convenience we suppress this in our notation. We introduce a notation for unions of consecutive intervals $I_s$ as follows: 
\begin{equation}
    \label{E:IklDef}
    I_{k,l} = \bigcup_{s=k}^l I_s.
\end{equation}
As before, we set \[R = \log^+(1/w).\]

We will study a variant of Korenblum's linear program in \cite[Equation 3.6]{korenblum1977beurling}. Our program, with real unknowns $x_s$, $s = 1, \ldots, N$, is as follows:

\begin{equation}
\label{E:LinearProgram}
    \begin{dcases}
        -\nu_0(I_{k,l}) + \sum_{s=k}^l x_s \, \leq \epsilon \cdot h(|I_{k,l}|),\\
        \quad \quad \quad \quad \quad \sum_{s=k}^l x_s \leq \mathcal{M}(A \cdot h, R|I_{k,l}), \\
        \quad \quad \quad \quad \quad \sum_{s=1}^N x_s = 0.        
    \end{dcases}
\end{equation}

The notation $R|I$ stands for the restriction of $R$ to the interval $I$, interpreted as before to vanish outside of the interval $I$. The quantities $\epsilon$ and $A$ are positive numbers. The most significant difference between \eqref{E:LinearProgram} and the linear program in \cite{korenblum1977beurling} is the appearance of the Hausdorff functional $\mathcal{M}(A\cdot h,R|I_{k,l})$ in our version, in the place where in Korenblum's original program the quantity $h(|I_{k,l}|)$ is present. In accordance with the presentation of Korenblum's theorem in the book by Hedenmalm, Korenblum and Zhu in \cite{hedenmalmbergmanspaces}, we will say that the program is \textit{consistent} if for every choice of $\epsilon > 0$ there exists a choice of $A = A(\epsilon) > 0$ for which the linear program \eqref{E:LinearProgram} has a solution $\{x_s\}_{s=1}^N$ for \textit{any} $N > 0$. Otherwise, the program is \textit{inconsistent}, in which case there exists a fixed $\epsilon > 0$ such that for any choice of $A > 0$ the program fails to have a solution for some $N = N(A)$. 

Our purpose in the rest of the section is to establish two claims.
The first claim is that if the program \eqref{E:LinearProgram} is consistent, then $S_{\nu}$ is cyclic in the corresponding $\Po^t(\mu)$ space. The second claim is that if \eqref{E:LinearProgram} is inconsistent, then $\nu(E) > 0$ for some set $E \in \assoch{h}{w}$. If the claims are established, then it follows from the first claim that if $S_{\nu}$ is not cyclic in $\Po^t(\mu)$, then the program is inconsistent, and consequently from the second claim we obtain $\nu(E) > 0$ for some $E \in \assoch{h}{w}$. Thus \thref{P:CyclicityProp} and the sufficiency part of \thref{T:CyclicityMainTheorem} are consequences of these two claims. Our proofs will use adaptations of Korenblum's technique from \cite{korenblum1977beurling}, with the main new ideas relating to the way in which the functional $\mathcal{M}(h,R)$ enters the picture.

\subsection{Sufficiency of consistency for cyclicity}

We are assuming that the program is consistent. We fix a small $\epsilon > 0$ and are given a corresponding $A = A(\epsilon) > 0$ for which the system in \eqref{E:LinearProgram} has a solution $\{x_s\}_{s=1}^N$ for every $N$. 

\subsubsection{Construction of a partial solution}
From a solution $\{x_s\}_{s=1}^N$ we construct a bounded function $f_{\epsilon,N}: \T \to \mathbb{R}$ according to the formula \begin{equation}
    \label{E:fepsilonNstructure} f_{\epsilon,N} = \nu(\T) + \sum_{s=1}^N g_s,
\end{equation} where $g_s$ is supported on the interval $I_s$. If $x_s \leq 0$, we set $g_s$ to be constant on $I_s$ so that $\int_{I_s} g_s \, d\m = x_s$. If $x_s > 0$, then the second inequality in \eqref{E:LinearProgram} implies that \[ x_s \leq \mathcal{M}(A(\epsilon)\cdot h, R|I_s),\] and so
by \thref{P:DualityThm} there exists a bounded non-negative function $g^*_s$ supported on $I_s$ which satisfies the pointwise inequality $g^*_s \leq R$, the upper estimate $\int_I g^*_s \, d\m \leq A(\epsilon) \cdot h(|I|)$ for all intervals $I \subset \T$, and the equality $\int_{I_s}6 g^*_s \, d\m = x_s$. We set $g_s := 6 g^*_s$. Then, for every $s \in \{1, \ldots, N\}$ we have that

\begin{itemize} 
    \item $g_s \leq 6R$ pointwise on $\T$,
    \item $\int_I g_s \, d\m \leq 6  A(\epsilon) \cdot h(|I|)$ for every interval $I \subset \T$,
    \item $\int_{I_s} g_s \, d\m = x_s.$
\end{itemize}

We use the last two of the above three properties to establish the following lemma.

\begin{lem}
    \thlabel{L:Lemma1Est} In the notation as above, for every interval $I \subset \T$ we have the estimate
    \[ \int_I f_{\epsilon,N} \, d\m \leq 26 A(\epsilon) \cdot h(|I|) + \nu(\T)|I|\]
\end{lem}

\begin{proof}
    Assume for the moment that $I$ does not include the point $1 \in \T$ in its interior. Let $\{I_s\}_{s=k}^l$ be those intervals in the partition of $\T$ corresponding to $N$ that intersect $I$. For the first interval $I_k$, the second point above implies that we have the estimate
    \begin{align*}
    \int_{I_k \cap I} f_{\epsilon,N} \, d\m & = \int_{I_k \cap I} g_k \, d\m + \nu(\T)|I_k \cap I| \\ &\leq  6A(\epsilon) \cdot h(|I \cap I_k|) + \nu(\T)|I_k \cap I|\\ &\leq  6A(\epsilon) \cdot h(|I|) + \nu(\T)|I_k \cap I|.
    \end{align*} For the last interval $I_l$, we similarly have 
    \[ \int_{I_l \cap I} f_{\epsilon,N} \, d\m \leq 6 A(\epsilon) h(|I|) + \nu(\T)|I_l \cap I|.\]
    For the union of intervals $I_s$ contained inside $I$ (if any), we use the third point above and the second inequality in \eqref{E:LinearProgram} to obtain
    \begin{align*}
        \int_{I_{k+1, l-1}} f_{\epsilon,N} \, d\m &= \sum_{s=k+1}^{l-1} x_s + \nu(\T)| I_{k+1,l-1}|\\ &\leq \mathcal{M}(A(\epsilon) \cdot h, R| I_{k+1, l-1}) + \nu(\T)| I_{k+1,l-1}| \\ &\leq A(\epsilon) \cdot h(I) + \nu(\T)| I_{k+1,l-1}|,
    \end{align*}  the last inequality following easily from the definition of the Hausdorff functional in \eqref{E:HausdorffFuncDef}. By adding our three estimates, we obtain that $\int_I f_{\epsilon,N} \, d\m \leq 13 A(\epsilon) \cdot h(|I|) + \nu(\T)|I|$, if $I$ does not contain the point $1$ in its interior. If it does, then we may express $I$ as a union of two intervals, none of which contains the point $1$ in its interior, and apply the already established estimate to the two parts.
\end{proof}

We note that in the estimate in \thref{L:Lemma1Est} it is $A(\epsilon)$ that is the dominating quantity, and that we lack control on its size. We may therefore think of the right-hand side in the inequality in \thref{L:Lemma1Est} as $A'(\epsilon)\cdot h(|I|)$, where $A'(\epsilon)$ is a large quantity comparable to $A(\epsilon)$.

\subsubsection{Limiting argument} At this point we will introduce a premeasure which is a weak-type limit of a subsequence of $\{f_{\epsilon, N}\}_N$. Premeasures are set functions defined on the algebra of intervals of $\T$. Their basic properties are explained in Appendix \ref{sec:appendixA} and will be used frequently below. 

By the definition of $f_{\epsilon,N}$ in \eqref{E:fepsilonNstructure}, properties of $g_s$ listed just before \thref{L:Lemma1Est}, and the third equation in \eqref{E:LinearProgram}, the sequence $\{f_{\epsilon,N}\}_N$ differs from a sequence of normalized premeasures only by the constant $\nu(\T) d\m$ (see \thref{D:PremDef} in Appendix~\ref{sec:appendixA}). Applying \thref{L:Helly} combined with \thref{L:Lemma1Est} we conclude that we may let $N \to \infty$ along a sequence and assume that we have the convergence 
\begin{equation}
\label{E:PremeasureWeakConv} f_{\epsilon,N} - \nu(\T)d\m \to f_\epsilon - \nu(\T)d\m
\end{equation} in the sense stated in \thref{L:Helly}. Here $f_\epsilon$ is a premeasure and not a function, and $f_\epsilon(\T) = \nu(\T)$.

\begin{lem} \thlabel{L:Lemma2Est}
In the notation as above, we have the estimate
\begin{equation}
    \label{E:nuFsmoothnessEstimate}
-\nu(I) + f_\epsilon(I) \leq \epsilon \cdot h(|I|) \end{equation} for each interval $I \subset \T$.
\end{lem}

The corresponding statement in Korenblum's paper \cite{korenblum1977beurling} is left without proof. We fill in the details, which requires us to use some specific properties of premeasures stated in the Appendix~\ref{sec:appendixA}.

\begin{proof}[Proof of \thref{L:Lemma2Est}]
The estimate follows from the convergence in \eqref{E:PremeasureWeakConv}. Without loss of generality, we may assume that the partitions of $\T$ associated to larger $N$ are refinements of those associated to smaller $N$. It follows that for any $\delta > 0$ there exists a half-open interval $I'$ which is a union of intervals in any partition $\{I_s\}_{s=1}^N$ associated to $N$ large enough, for which the following four statements hold:

\begin{itemize}
    \item the jumps $J_{\rho_\sigma}(t)$, $J_{\rho_\sigma}(t')$ defined in \eqref{E:JumpDef} corresponding to the endpoints $e^{it}, e^{it'}$ of $I'$ are smaller than $\delta$,
    \item $|f_\epsilon(I') - f_\epsilon(I)| \leq \delta$,
    \item $|\nu(I') - \nu(I)| < \delta$,
    \item $\big| |I'| - |I|\big| < \delta$.
\end{itemize}
All four parts can be ensured by choosing $I'$ to be the interval obtained from $I$ by a small perturbation of its endpoints. The first point can be ensured by \thref{L:FiniteJumpLemma}, and the second and third points can be ensured by properties (P2) and (P3) of premeasures stated in \thref{D:PremDef}. Then from the first inequality in \eqref{E:LinearProgram}, our definitions \eqref{E:nu0def} and \eqref{E:fepsilonNstructure}, the pointwise convergence property in \eqref{E:AlmostConvJumpError}, and the first bullet point above, we obtain
\begin{align*}
-\nu(I') + f_\epsilon(I') &= - \nu(I') + \rho_{f_\epsilon}(t') - \rho_{f_\epsilon}(t) \\ 
&\leq - \nu(I') + \limsup_{N \to \infty} \Big(  \rho_{f_{\epsilon, N}}(t') - \rho_{f_{\epsilon, N}}(t) \Big) + 2\delta \\
&=  -\nu(I') + \limsup_{N \to \infty} \int_{I'} f_{\epsilon,N}\, d\m + 2\delta \\ &=  -\nu_0(I') + \limsup_{N \to \infty} \sum_{s: I_s \subset I'} x_s + 2 \delta \\&\leq  \epsilon \cdot h(|I'|) + 2\delta. \end{align*} It follows from the last three bullet points above that \[ -\nu(I) + f_{\epsilon}(I) \leq 4 \delta + \epsilon\cdot h(|I| + \delta).
\]
By letting $\delta \to 0$ we obtain \eqref{E:nuFsmoothnessEstimate}.
\end{proof}

\subsubsection{A candidate solution} \label{S:interludiumSubsec}

In order to prove cyclicity of $S_{\nu}$ we should construct a candidate function $F$ for which $S_{\nu}F$ approximates the constant function $1$ in $\Po^t(\mu)$. Set \begin{equation}
    \label{E:FepsilonDef}
F_\epsilon(z) := \exp \Big(\int_{[0,2\pi)} \frac{e^{it} + z}{e^{it} - z} df_\epsilon(e^{it})\Big), \quad z \in \D,\end{equation} the integral being defined in the sense of premeasures as in \eqref{E:IntegratePremeasure} and the following discussion. We have \[|S_{\nu}(z)F_\epsilon(z)| = \exp \Big(\int_{[0,2\pi)} P_z(e^{it}) d[-\nu +f_\epsilon](e^{it})\Big), \quad z \in \D, \] where $P_z$ is the Poisson kernel in \eqref{E:PoissonKernelDef}. From a combination of \thref{L:Lemma2Est} and the Poisson integral estimate \eqref{E:PoissonEstimate} in Appendix A, we obtain 
\begin{equation}
    \label{E:SnuFepsilonGrowthDiskEstGeneral}
    |S_{\nu}(z)F_\epsilon(z)| \leq \exp \Big( C\epsilon \frac{h(1-|z|)}{1-|z|} \Big), \quad z \in \D,
\end{equation}
where $C > 0$ is a constant independent of $\epsilon$. This implies that \[ \limsup_{\epsilon \to 0} |S_{\nu}(z)F_\epsilon(z)| \leq 1, \quad z \in \D \] and that \[ \limsup_{\epsilon \to 0} \int_\D |S_{\nu}F_\epsilon|^t G_{ah} dA < \infty\] for any $a > 0$ (recall the definition of $G_{ah}$ in \eqref{E:MuGeneralizedDef}). Since $S_\nu(0)F_\epsilon(0) = 1$, it follows that as $\epsilon \to 0$, the functions $S_{\nu}(z)F_\epsilon(z)$ converge to $1$ uniformly on compact subsets of $\D$ (more precisely, we need to let $\epsilon \to 0$ along an appropriate sequence). To conclude that $S_{\nu}$ is cyclic in $\Po^t(\mu)$, we still lack something. Firstly, it is not clear if we have a norm bound on $S_{\nu} F_\epsilon$ in $\Po^t(\mu)$, since the measure $\mu$ involves a part on the boundary $\T$ also. Secondly, we do not know if $S_{\nu}F_\epsilon$ is a member of $[S_\nu]$, the $M_z$-invariant subspace generated by $S_{\nu}$, or even if it is a member of $\Po^t(\mu)$.

\subsubsection{Absorption into the invariant subspace}
We may assume that $6t$ is an integer by the first remark following \thref{P:CyclicityProp}. We are also assuming that $t > 1$, so that $\Po^t(\mu)$ is reflexive, and in particular, any norm bounded sequence of functions in the space has a subsequence converging weakly to an element of the space. The analyticity assumption on $\Po^t(\mu)$ implies that any $f \in \Po^t(\mu)$ is uniquely determined by $f|\D$, and will otherwise not play any role in the proofs.

According to the second remark following \thref{P:CyclicityProp}, it will suffice to show that $S_{\nu}^{1/6t} = S_{\nu/6t}$ is cyclic in $\Po^t(\mu)$. To ease the notation, we divide all the appearing measures and premeasures by the factor $6t$, and make the notational replacements $\nu \to \nu/6t$, $f_{\epsilon,N} \to f_{\epsilon,N}/6t$ and $f_{\epsilon} \to f_{\epsilon}/6t$. Since $6t > 0$, all our previous estimates are valid (such as the ones in \thref{L:Lemma1Est} and \thref{L:Lemma2Est}) with new constants comparable to the old ones. In particular, after this notational replacement, we now have on $\T$ the pointwise estimate \begin{equation} \label{E:FepsilonNPointwiseEst}
f_{\epsilon,N} \leq R/t + \nu(\T)/6t \leq R/t + 1 \end{equation} (recall \eqref{E:fepsilonNstructure}, the listed properties of $g_s$ right before the statement of \thref{L:Lemma1Est}, and increase $t$ slightly to make $\nu(T)/6t \leq 1$).

Set \begin{equation}
    \label{E:FEpsilonNdef}
    F_{\epsilon,N}(z) := \exp \Big( \int_{[0,2\pi)} \frac{e^{it} + z}{e^{it} - z} f_{\epsilon,N}(e^{it}) d\m(e^{it})\Big), \quad z \in \D.
\end{equation} We note that, since the functions $f_{\epsilon,N}$ are bounded, the sequence $\{F_{\epsilon,N}\}_N$ consists of functions in $\hil^\infty$, and so in particular $F_{\epsilon,N} \in \Po^t(\mu)$. We have also the pointwise convergence $F_{\epsilon,N}(z) \to F_\epsilon(z)$ for $z \in \D$ by the convergence $f_{\epsilon,N} \to f_\epsilon$ in the sense of \thref{L:Helly}.

We establish now some estimates uniform in $N$. From \thref{L:Lemma1Est} and the estimate \eqref{E:PoissonEstimate} we deduce
\begin{equation} \label{E:FepsilonNGrowthEst}
    |F_{\epsilon,N}(z)| \leq \Big( A'(\epsilon) \frac{h(1-|z|)}{1-|z|)}\Big), \quad z \in \D
\end{equation} where $A'(\epsilon)$ is some large positive constant comparable to $A(\epsilon)$. In the case $\mu$ has form \eqref{E:MuGeneralizedDef} for some $a > 0$, take a positive integer $D = D(\epsilon)$ so large that we have 
\begin{equation}
    \label{E:FEpsilonRootDEstGen} |F_{\epsilon,N}(z)|^{t/D} \leq \exp \Big( (a/2) \frac{h(1-|z|)}{1-|z|)} \Big), \quad z \in \D.
\end{equation}
In the case that $\mu$ has the form \eqref{E:MuDef} for some $\alpha \in (-1,0]$ (and so $h$ is given by \eqref{E:CarlesonGaugeDef}), fix a small $\delta > 0$ for which we have $\alpha - 2\delta > -1$, and let $D = D(\epsilon) > 0$ be a positive integer so large that \begin{align}
    \label{E:FEpsilonRootDEst} |F_{\epsilon,N}(z)|^{t/D} &= \exp \Big( (A'(\epsilon)/D) \frac{h(1-|z|)}{(1-|z|)} \Big) \\ &= \Big(\frac{e}{1-|z|} \Big)^{A'(\epsilon)/D} \nonumber \\ & \leq \Big(\frac{e}{1-|z|} \Big)^\delta, \quad z \in \D. \nonumber
\end{align}

The fact that $S_\nu F_\epsilon$ is contained in the invariant subspace $[S_\nu]$ can be established by using the above estimates and a standard technique of absorbing the function $S_\nu F_\epsilon$ into $[S_\nu]$ piece by piece.

\begin{lem} 
    In the notation as above, we have that $S_{\nu}F_\epsilon^{k/D}$ is contained in $[S_\nu]$ for each integer $k=0, 1, \ldots, D$. In particular, $S_{\nu}F_\epsilon$ is contained in $[S_\nu]$. Moreover, we have
    \[ \limsup_{\epsilon \to 0} \|S_{\nu}F_{\epsilon}\|_{\mu,t} < \infty.\]
\end{lem}

\begin{proof}
    We first carry out the proof in the context of measures of the form \eqref{E:MuGeneralizedDef}, the other case being similar and treated afterwards.
    The proof will proceed by induction on $k$, with the base case $k=0$, corresponding to $S_\nu \in [S_\nu]$, being trivial. Let us then assume that $k < D$ and that the claim holds for $k$. We will show that it holds also for $k+1$. 

    Using the inequalities $|S_\nu(z)| \leq |S_\nu(z)|^{k/D}$, $C \epsilon tk/D < a/2$ (which holds for sufficiently small $\epsilon$) and \eqref{E:SnuFepsilonGrowthDiskEstGeneral}, we obtain, for $z \in \D$, the estimate
    \begin{align*}
        |S_{\nu}(z) F_\epsilon^{k/D}(z)|^t &\leq |S_{\nu}(z) F_\epsilon(z)|^{tk/D} \\
        &\leq \exp \Big( (C \epsilon tk/D)\frac{h(1-|z|)}{1-|z|} \Big), \\
        & \leq \exp \Big( (a/2) \frac{h(1-|z|)}{1-|z|} \Big) \\
        &= G_{ah}(z)^{-1/2}.
    \end{align*}
    Combined with \eqref{E:FEpsilonRootDEstGen}, we conclude that for small $\epsilon$ we have the bound 
\begin{equation}
\label{E:SnuFepsDbound} |S_{\nu}(z) F_\epsilon^{k/D}(z)F_{\epsilon,N}^{1/D}(z)|^t G_{ah}(z) \leq 1 
\end{equation} which gives uniform norm boundedness of the sequence $\{S_\nu F^{k/D}_\epsilon F^{1/D}_{\epsilon,N}\}_N$ in $\mathcal{L}^t(G_{ah}dA)$. The norm bound is independent of $\epsilon$, as long as $\epsilon$ is sufficiently small. 
Let now $C = \{ z \in \T : w(z) > 0 \}$ be the natural carrier set of $w$, and introduce the set \[\mathcal{C} := \{G \in \Po^t(\mu) : |G| \leq \exp(R/Dt + 1/D) \, \m\text{-a.e on } C\}.\] It is easy to see that $\mathcal{C}$ is convex. It is also closed in the norm topology of $\Po^t(\mu)$, since if $\{G_n\}_n$ is a sequence in $\mathcal{C}$ converging in norm to $G$, then a subsequence $\{G_{n_k}\}_k$ may be extracted which converges pointwise to $G$ on a set of full $\m$-measure in the carrier $C$. It follows that $\mathcal{C}$ is weakly closed, by basic functional analysis. From the definition of $F_{\epsilon,N}$ in \eqref{E:FEpsilonNdef} and the inequality \eqref{E:FepsilonNPointwiseEst} we deduce that $F^{1/D}_{\epsilon,N} \in \mathcal{C}$. Keeping in mind also that $R = \log^+(1/w)$, and that we have the estimate \eqref{E:FEpsilonRootDEstGen}, we deduce easily that $\{F^{1/D}_{\epsilon, N}\}_N$ is a norm bounded sequence in $\Po^t(\mu)$.  Therefore $F^{1/D}_{\epsilon} \in \mathcal{C}$, since $F^{1/D}_{\epsilon}$ is a pointwise limit in $\D$ of $\{F^{1/D}_{\epsilon, N}\}_N$, and any weakly convergent subsequence of $\{F^{1/D}_{\epsilon, N}\}_N$ must necessarily converge to an element of $\mathcal{C}$. On $\T$ we have $|S_{\nu}| = 1$, and it follows from the membership of $F^{1/D}_{\epsilon,N}$ and $F^{1/D}_{\epsilon}$ in $\mathcal{C}$ that on the carrier set $C$ of $w$ we have the pointwise inequalities 
    \begin{align*}
        |S_{\nu}F_\epsilon^{k/D} F^{1/D}_{\epsilon,N}|^tw &\leq e^{t} \cdot \exp(R(k+1)/D)w \\
        &\leq e^{t} (1+w^{-(k+1)/D})w \\
        &\leq e^{t} (1 + 2w) \in \mathcal{L}^1(d\m).
    \end{align*}
    By combining the above estimate with \eqref{E:SnuFepsDbound}, we obtain 
    \begin{equation}
        \label{E:UniformNormFepsilonEstimate}
        \sup_{N} \| S_{\nu}F^{k/D}_{\epsilon} F_{\epsilon, N}^{1/D}\|_{\mu,t} < \infty
    \end{equation} with the estimate being uniform also in $\epsilon$ small enough. Since $F^{1/D}_{\epsilon, N} \in \hil^\infty$, and by induction hypothesis $S_{\nu}F^{k/D}_{\epsilon}$ is contained in $[S_\nu]$, by weak compactness of $\Po^t(\mu)$ (recall $t > 1$) we obtain that $S_{\nu}F^{(k+1)/D}_{\epsilon}$ is also contained in $[S_\nu]$, with norm $\|S_{\nu}F^{(k+1)/D}_{\epsilon}\|_{\mu,t}$ not larger than the supremum on the left-hand side in \eqref{E:UniformNormFepsilonEstimate}.

    The proof in the case \eqref{E:MuDef} and $\alpha \in (-1,0]$ is similar. By \eqref{E:SnuFepsilonGrowthDiskEstGeneral}, and an argument similar to the one given at the beginning of this proof, we obtain for $\epsilon$ small enough the estimate
    \[ |S_{\nu}(z) F_\epsilon^{k/D}(z)|^t \leq \Big(\frac{e}{1-|z|}\Big)^\delta, \quad z \in \D.\] Hence by \eqref{E:FEpsilonRootDEst} it holds that 
    \[ |S_{\nu}(z) F_\epsilon^{k/D}(z)F_{\epsilon,N}^{1/D}(z)|^t (1-|z|)^\alpha \leq 3 \cdot (1-|z|)^{\alpha - 2\delta} \in \mathcal{L}^1(dA),
    \] where the last inclusion holds as a consequence of the inequality $\alpha - 2\delta > -1$. Thus we have also in this case the corresponding norm boundedness of $\{S_\nu F^{k/D}_\epsilon F^{1/D}_{\epsilon,N}\}_N$ in $\mathcal{L}^t(dA_\alpha)$. The rest of the proof is the same as in the case of measures of the form \eqref{E:MuGeneralizedDef}.
\end{proof}

According to the discussion in Section \ref{S:interludiumSubsec}, we have now shown that consistency of the linear program \eqref{E:LinearProgram} implies the cyclicity of the singular inner function $S_{\nu}$ in $\Po^t(\mu)$. Indeed, we conclude from that discussion and the above lemma that for small $\epsilon > 0$ the norm bounded family $\{S_\nu F_\epsilon\}_\epsilon$ of members of $[S_\nu]$ has the constant function $1$ as a weak cluster point. So $1 \in [S_\nu]$, and $S_\nu$ is cyclic.

\subsection{Necessity of consistency for cyclicity}

According to what was said in the last paragraph of Section \ref{S:KorenblumProgramSubsec}, in order to complete the proofs of \thref{T:CyclicityMainTheorem} and \thref{P:CyclicityProp}, we now need to show that if the linear program \eqref{E:LinearProgram} is inconsistent, then there exists an associated set $E$ for which $\nu(E) > 0$. 

\subsubsection{The Inconsistency Inequality} \label{S:InconsistencyIneqSec}

We start the necessity proof in the same way as Korenblum started his proof in \cite{korenblum1977beurling} by deriving an inequality from the inconsistency of the linear program. 

Note that the first two inequalities in \eqref{E:LinearProgram} may be written as a single inequality in the following way:
\[
-\nu_0(I_{k,l}) + \sum_{s=k}^l x_s \, \leq \min \big(-\nu_0(I_{k,l}) + \mathcal{M}(A \cdot h, R|I_{k,l}), \, \epsilon \cdot h(|I_{k,l}|) \big).
\]
Only for a moment, let us denote the right-hand side of this inequality by $b_{k,l}$, and think of $y_s = -\nu_0(I_s) + x_s$ as the unknowns in a new linear program. Note that $\sum_{s=1}^N y_s = 0$ if and only if $\sum_{s=1}^N x_s = 0$, since $\nu_0$ is normalized and additive. The proof of the following duality lemma for linear programs appears in \cite[Lemma 7.7]{hedenmalmbergmanspaces}, where Korenblum's work on cyclic vectors is presented. The lemma appears implicitly in the work of Korenblum in \cite{korenblum1977beurling}.
\begin{lem}
    \thlabel{L:DualityLemma1NecProof}
    Let $N$ be a fixed positive integer, and $b_{k,l}$ be real numbers, $1 \leq k \leq l \leq N$.
The system of inequalities \[ \sum_{s=k}^l y_s \leq b_{k,l}\] subject to the constraint $\sum_{s=1}^N y_s = 0$ is solvable if and only if \[\sum_{i=1}^m b_{k_i, k_{i+1}-1} \geq 0\]
for every increasing sequence $\{k_i\}_{i=1}^{m-1}$ of integers,  $k_i \in [1, \ldots,  N]$, where $k_1 = 1$ and $k_m = N+1$.
\end{lem}

Note that the sequence $\{k_i\}_{i=1}^m$ in the statement of the lemma partitions the set of integers $[1, \ldots, N]$ into $m-1$ "intervals" $[k_i, k_{i+1}-1]$. Such partitions are called \textit{simple coverings} in \cite{hedenmalmbergmanspaces} and \cite{korenblum1977beurling}. A simple covering corresponds to a covering of the circle $\T$ by a finite number of half-open intervals $\{J_i\}_{i=1}^{m-1}$ as in \eqref{E:IklDef}, where we use the notation \[J_i = I_{k_i, k_{i+1}-1}.\] Then by our above discussion and \thref{L:DualityLemma1NecProof} we see that inconsistency of the system \eqref{E:LinearProgram} implies that there exists $\epsilon > 0$ such that to every choice of $A$ there corresponds a covering $\{J_i\}_i$ of $\T$ consisting of half-open intervals, which satisfies
\begin{equation}
    \label{E:PreIncIneq}\sum_{i} \min \big(-\nu_0(J_i) + \mathcal{M}(A \cdot h, R|J_i), \, \epsilon \cdot h(|J_i|) \big) < 0.
\end{equation} Denoting by $J'_i$ those intervals among $J_i$ for which \[\min \big(-\nu_0(J_i) + \mathcal{M}(A \cdot h, R|J_i), \, \epsilon \cdot h(|J_i|) \big) = -\nu_0(J_i) + \mathcal{M}(A \cdot h, R|J_i)\] and by $J''_i$ the remaining ones, by rearranging \eqref{E:PreIncIneq} we obtain  our \textit{Inconsistency Inequality}, which reads 

\begin{equation}
    \label{E:InconsistencyIneq}
    \sum_{i} \mathcal{M}(A \cdot h, R|J'_i)  + \sum_{i} \epsilon \cdot h(|J''_i|) < \sum_{i} \nu_0(J'_i).
\end{equation}

We emphasize that $\epsilon$ remains now fixed for the rest of the proof, and that the covering $\{J_i\}_i$ depends on the choice of $A$, but that we suppress this in our notation for now. Note also that the left-hand side of the inequality is a positive quantity. In fact, this quantity stays away from $0$.

\begin{lem}
    \thlabel{L:InconsistencyLowerBoundLemma} Assume that $R > 0$ on a subset of positive Lebesgue measure on $\T$. Then we have 
    \[ \inf_{A : A \geq \epsilon} \, \sum_{i} \mathcal{M}(A \cdot h, R|J'_i) + \sum_{i} \epsilon \cdot h(|J''_i|) > c_0\] for some positive number $c_0$.
\end{lem}

\begin{proof}
Let us assume that by varying $A$ it is possible to have the left-hand side above converge to zero. We will show $R$ vanishes almost everywhere with respect to the Lebesgue measure $d\m$.

By the definition in \eqref{E:HausdorffFuncDef}, for each interval $J'_i$ there exists a choice of an open set $U'_i \subset J'_i$ which is a union of a family of open intervals $\{\ell_{i,k}\}_k$ for which $\int_{J'_i \setminus U'_i} R \, d\m + \sum_{ k } A\cdot h(|\ell_{i,k}|)$ is as close as we wish to $\mathcal{M}(A\cdot h, R|J'_i)$. Let 
\[J'_A = \bigcup_i J'_i, \quad U'_A = \bigcup_i U'_i \] and let $\{\ell_k\}_k$ be the union over $i$ of the families $\{\ell_{i,k}\}_k$, so that $U'_A = \bigcup_k \ell_k$. Since $A \geq \epsilon$, by our assumption we may arrange our choices so that 
\begin{equation}
\label{E:ConvToZeroQuantity}
\int_{J'_A \setminus U'_A} R \, d\m + \epsilon \cdot \Big( \sum_{k} h(|\ell_k|) + \sum_{i} h(|J''_i|) \Big) \to 0.
\end{equation}  Since $\epsilon$ is fixed, we conclude that \[\sum_k h(|\ell_k|) \to 0, \quad \sum_i h(|J''_i|) \to 0,\] and in particular we must have \[|U'_A| \leq \sum_k |\ell_k| \to 0, \quad \sum_{i} |J''_i| \to 0.\] Since $\T = (J'_A \setminus U'_A) \cup U'_A \cup  \Big( \cup_i J''_i \Big)$, we see that the Lebesgue measures of $J'_A \setminus U'_A$ must grow to full $\m$-measure in $\T$. Since \eqref{E:ConvToZeroQuantity} implies that $\int_{J'_A \setminus U'_A} R \, d\m \to 0$, and since $R$ is non-negative, we conclude that $\int_\T R \, d\m = 0$. So $R \equiv 0$ almost everywhere with respect to $d\m$.
\end{proof}

We mention again that the case $R = \log^+(1/w) = 0$ is irrelevant for us, as it is equivalent to $w \geq 1$ in which case $\int_\T \log w \, d\m > -\infty$, and so the space $\Po^t(\mu)$ has the simple structure explained in Section \ref{S:MainResSect}. We may therefore assume that $R \neq 0$, and that the lower bound in \thref{L:InconsistencyLowerBoundLemma} holds.

\subsubsection{Analysis of the inequality}
We have seen that inconsistency of \eqref{E:LinearProgram} leads us to the conclusion that there exist an $\epsilon > 0$ such that to each choice of $A$ larger than $\epsilon$, there correspond two families of disjoint half-open intervals $\{J'_i\}_i$ and $\{J''_i\}_i$, which together cover the unit circle $\T$, and for which the inequality \eqref{E:InconsistencyIneq} holds, with the left-hand side being bounded away from $0$.

Let $E_A$ be the union of closures of the intervals $J'_i$:
\[ E_A := \bigcup_i \text{clos}(J'_i).\] Since $\{J'_i\}_i$ is a finite family, the set $E_A$ is closed and differs from the union $\bigcup_i J'_i$ only by finitely many points. Using the notation from the proof of \thref{L:InconsistencyLowerBoundLemma}, the inconsistency inequality \eqref{E:InconsistencyIneq} implies that there exist open sets $U'_i = \bigcup_{k} \ell_{i,k}$ contained in $J'_i$ for which
\[ c_0 < \sum_{J'_i} \Bigg( \int_{J'_i \setminus U'_i} R \, d\m + \sum_{k} A \cdot h(|\ell_{i,k}|) \Bigg) + \sum_{i} \epsilon \cdot h(|J''_i|) < \sum_{i} \nu_0(J'_i).\]
Since $\nu_0 = \nu - \nu(\T)d\m \leq \nu$, $\nu$ is a non-negative measure, and $\{J'_i\}_i$ is a disjoint family of intervals contained in $E_A$, the right-most quantity above is dominated by $\nu(E_A)$. Using still the notation from the proof of \thref{L:InconsistencyLowerBoundLemma} and setting \[E'_A = E_A \setminus U'_A,\] we obtain
\begin{equation}
 \label{E:RefinedInconsistencyIneq}
    c_0 < \int_{E'_A} R \, d\m + \sum_{k} A \cdot h(|\ell_k|) + \sum_{i} \epsilon \cdot h(|J''_i|) < \nu(E_A).
\end{equation}

The following observations are consequences of the second of the above inequalities.

\begin{itemize}
    \item The measure $\nu$ is non-negative, so $\nu(E_A) \leq \nu(\T)$. As $A \to \infty$, we must therefore have $\sum_{k} |\ell_k| \to 0$. Since $U'_A = \cup_k \ell_k$, we deduce that 
    \begin{equation}
        \label{E:EaE'relation}
        E_A = E'_A + U'_A
    \end{equation}
    where $|U'_A| \to 0$ as $A \to \infty$. In particular, it follows that \[ \lim_{A \to \infty} |E_A| - |E'_A| = 0.\]
    \item The family of complementary intervals to $E_A$ is $\{\text{int}(J''_i)\}_i$, the interiors of the intervals $\{J''_i\}_i$, while the complementary intervals of $E'_A$ consists of the union of the families $\{\text{int}(J''_i)\}_i$ and $\{\ell_k\}_k$. From \eqref{E:RefinedInconsistencyIneq} we deduce that
    \begin{equation}
        \label{E:CarlesonSupBBoundEEPrime}
        \sup_{A : A \geq \epsilon}  \Big( \sum_{k}  h(|\ell_k|) + \sum_{J''_i}  h(|\text{int}(J''_i)|) \Big)  < \infty.
    \end{equation} 
\end{itemize}

\subsubsection{Compactness properties of the class $\BCh$}

For a closed subset $E$ of $\T$ and $\delta > 0$, let $E^\delta$ be the usual open $\delta$-neighbourhood of $E$, consisting of all points at distance less than $\delta$ from $E$:
\[ E^\delta := \{ z \in \T : \dist{z}{E} < \delta \}.\]

\begin{lem}
    \thlabel{L:FinalLemmaNecessity} Let $\{E_n\}_n$ be a sequence of sets in $\BCh$. Let $\{\ell_{n,k}\}_k$ be the sequence of maximal open intervals in $\T$ complementary to $E_n$, and assume that 
    \begin{equation}
        \label{E:UniformEnCarlesonBound}
        \sup_n \sum_{k} h(|\ell_{n,k}|) < \infty.
    \end{equation}
    Assume further that for each $n$ the sequences $\{\ell_{n,k}\}_k$ are ordered such that the lengths of the intervals are decreasing in $k$:
    \[ |\ell_{n,1}| \geq |\ell_{n,2}| \geq \ldots,\]
    Then there exists a set $E_\infty \in \BCh$ and a subsequence $\{E_{n_j}\}_j$ such that $E_{n_j} \to E_\infty$ in the following sense:
    \begin{enumerate}[(i)]
    \item $\lim_{j \to \infty} |E_{n_j}| = |E_\infty|$,
    \item for every $\delta > 0$ we have that $E_\infty \subset E^\delta_{n_j}$ and $E_{n_j} \subset E^\delta_\infty$ for all sufficiently large $j$.
    \item For each maximal open interval $\ell_k$ complementary to $E_\infty$, we have that $\ell_{n_j,k}$ converge to $\ell_k$ as $j \to \infty$, in the sense that the endpoints of $\ell_{n_j,k}$ converge to the endpoints of $\ell_k$.
    \end{enumerate}
\end{lem}

The lemma is a direct generalization to sets of positive Lebesgue measure of \cite[Lemma 3.1.1]{korenblum1977beurling}. See also \cite[Lemma 7.6]{hedenmalmbergmanspaces}. In \cite{bergqvist2024distributing} a similar compactness statement appears in the context of unions of dyadic cubes. The proofs of these statements are all similar. We provide one for completeness.

\begin{proof}[Proof of \thref{L:FinalLemmaNecessity}]
    By an initial passing to a subsequence, we may assume that we have the convergence of Lebesgue measures of $E_n$, which we may express as the convergence of Lebesgue measures of the complements: 
    \begin{equation}
        \label{E:ComplementNormConvergence}
        \lim_{n \to \infty} \sum_{k} |\ell_{n,k}| = L \geq 0.
    \end{equation} 

    For $k=1, 2, \ldots$, we successively attempt to pass to a subsequence of $\{E_n\}_n$ to ensure that (after relabeling) $\ell_{n,k}$ converges in the sense of $(iii)$ above to some interval $\ell_k$ of positive length. Our process stops for some finite $K$ if $\lim_{n \to \infty} |\ell_{n,K}| = 0$, and then we set $\ell_{k} = \varnothing$ for $k \geq K$. Else, the process goes on indefinitely. In the latter case, by the usual diagonal subsequence argument we ensure that $\ell_{n,k}$ converges to $\ell_k$ as $n \to \infty$, for all $k$. We set in either case
    \[ E_\infty := \T \setminus \Big( \bigcup_{k} \ell_k \Big),\] with the appearing union being finite in the former case. Note that we have already ensured $(iii)$. 
    
    Since for each positive integer $M$ it holds that \[ \sum_{k \leq M} h(|\ell_k|) = \lim_{n \to \infty} \sum_{k \leq M} h(|\ell_{n,k}|) \leq \sup_n \sum_k h(|\ell_{n,k}|),\] by letting $M \to \infty$ we conclude from \eqref{E:UniformEnCarlesonBound} that $E_\infty \in \BCh$. 

    To establish that $\lim_{n \to \infty} |E_n| = |E_\infty|$ we may equivalently establish that
    \begin{equation}
        \label{E:IntermStepL1Conf}
        \sum_{k} |\ell_k| = \lim_{n \to \infty} \sum_k |\ell_{n,k}| = L.
    \end{equation}
    Fatou's lemma immediately shows that the left-hand side in \eqref{E:IntermStepL1Conf} is at most as large as the right-hand side, which equals $L$ by \eqref{E:ComplementNormConvergence}. As the left-hand side is non-negative, we are done if $L = 0$. Else, we assume for the sake of contradiction that $\sum_k |\ell_k| = L - \alpha > 0$ for some $\alpha > 0$.  Note that for every $K$, we have \begin{align}
     \label{E:InequalitiesForContradiction}
        \sum_{k \geq K} h(|\ell_{n,k}|) &= \sum_{k \geq K} \frac{h(|\ell_{n,k}|)}{|\ell_{n,k}|}|\ell_{n,k}| \\
        &\geq \frac{h(|\ell_{n,K}|)}{|\ell_{n,K}|} \sum_{k \geq K} |\ell_{n,k}|. \nonumber
    \end{align} The last inequality follows from our ordering assumption on $\{|\ell_{n,k}|\}_k$ and the assumption (R1) in Section~\ref{S:OptimProbSubsec} stating that $h(x)/x$ is decreasing in $x$. If our process terminated at $k = K$, then $\lim_{n \to \infty} |\ell_{n,K}| = 0$ and $\sum_{k < K} |\ell_k| = L - \alpha$. Since (R2) states that $\lim_{x \to 0} h(x)/x = \infty$, the bound \eqref{E:UniformEnCarlesonBound} and the inequality in \eqref{E:InequalitiesForContradiction} together imply that $\sum_{k \geq K} |\ell_{n,k}| \to 0$ as $n \to \infty$. But then
\begin{align*}
L &= \lim_{n \to \infty} \sum_{k < K} |\ell_{n,k}| + \sum_{k \geq K} |\ell_{n,k}| \\ &=  \sum_{k < K} |\ell_k| \\ &= L - \alpha < L,    
\end{align*} and we have reached a contradiction. If the process never terminated, the contradiction is reached similarly. Now, for each $K$ we have $\lim_{n \to \infty} |\ell_{n,K}| = |\ell_K| > 0$, and also
    \[ L = \sum_{k < K} |\ell_{k}| + \lim_{n \to \infty}  \sum_{k \geq K} |\ell_{n,k}|.\] Since $\sum_{k < K} |\ell_k| \leq L - \alpha$, it follows from \eqref{E:ComplementNormConvergence} that for every $K$ we have \[\lim_{n \to \infty} \sum_{k \geq K} |\ell_{n,k}| \geq \alpha.\] Letting $n \to \infty$ in $\eqref{E:InequalitiesForContradiction}$ we deduce
    \[ \sup_n \sum_{k} h(|\ell_{n,k}|) \geq \frac{h(|\ell_K|)}{|\ell_K|}\alpha\] Since $|\ell_K| \to 0$ as $K \to \infty$, we obtain again a contradiction, this time to \eqref{E:UniformEnCarlesonBound}. We have thus  established \eqref{E:IntermStepL1Conf}, and so $(i)$ holds.

    For an open interval $\ell$ in $\T$ and $\delta \leq |\ell|/2$, let $\ell(\delta)$ be the closed interval with same midpoint as $\ell$ but of length $|\ell| - 2\delta$. Note that $\ell(\delta)$ degenerates to a point if $\delta = |\ell|/2$. If $\delta > |\ell|/2$, let $\ell(\delta)$ be the empty set. With this notation, we have 
    \[ \T \setminus E^\delta_\infty = \bigcup_k \ell_k(\delta)\] and 
    \[ \T \setminus E^\delta_n = \bigcup_k \ell_{n,k}(\delta).\]
    Fix $\delta > 0$. To show that $E_n \subset E^\delta_\infty$ for large $n$, we may equivalently show that \[\bigcup_k \ell_{n,k} \supset \bigcup_k\ell_k(\delta)\] for all large $n$. Since $|\ell_k| \to 0$ as $k \to \infty$, the sets $\ell_k(\delta)$ are non-empty for only finitely many indices $k$. Since $\ell_k(\delta)$ is contained in the interior of $\ell_k$ and $\ell_{n,k} \to \ell_k$ as $n \to \infty$, we obtain easily that $\ell_k(\delta) \subset \ell_{n,k}$ for all large $n$, and all $k$. This establishes the inclusion $E_n \subset E^\delta_\infty$ for large $n$.
        
    The other inclusion, namely $E_\infty \subset E^\delta_n$ for all large $n$, is equivalent to \[\bigcup_k \ell_k \supset \bigcup_k\ell_{n,k}(\delta).\] The same argument works by virtue of
    \begin{equation}
        \label{E:UniformEllnkBound} \lim_{K \to \infty} \lim_{n \to \infty} \, |\ell_{n,K}| = 0,
    \end{equation} which holds since $\lim_{n \to \infty} |\ell_{n,K}| = |\ell_K|$ for all $K$. By the ordering assumption on $\{|\ell_{n,k}|\}_k$ there exists $K = K(\delta) > 0$ for which $\ell_{n,k}(\delta) = \varnothing$ for $k \geq K$ and all large $n$, and we finish the proof as in the case of the other inclusion.
\end{proof}

\subsubsection{Finishing the construction of the associated set}

We return to our inequality \eqref{E:RefinedInconsistencyIneq}. Since \eqref{E:CarlesonSupBBoundEEPrime} holds, there exists a set $E_A \in \BCh$ for which we have, along some sequence of $A \to \infty$, the convergence \[ E_A \to E_\infty\] in the sense of \thref{L:FinalLemmaNecessity}. In turn, by \eqref{E:CarlesonSupBBoundEEPrime} yet again, we may pass to a subsequence once more to ensure also that \[E'_A \to E'_\infty\] as $A \to \infty$, in the same sense. Note that $|E_\infty| = |E'_\infty|$ by part $(i)$ of \thref{L:FinalLemmaNecessity} and what was said in the first observation following \eqref{E:RefinedInconsistencyIneq}.

We shall show that $E_\infty$ is the desired set. Namely, it will be shown that $E_\infty \in \assoch{h}{w}$ and $\nu(E_\infty) > 0$. 

\begin{lem} \thlabel{L:EE'equalityLemma}
    With the notation as above, we have $E'_\infty \subset E_\infty$.
    In particular, since the Lebesgue measures of the two sets are equal, they differ at most by a set of Lebesgue measure zero.
\end{lem}

\begin{proof}
    We prove the reverse inclusion \[ \T \setminus E'_\infty \supset \T \setminus E_\infty.\]
    Let $z \in \T \setminus E_\infty$, and let $\ell$ be the maximal open interval complementary to $E_\infty$ which contains $z$. By part $(iii)$ of \thref{L:FinalLemmaNecessity}, there exists a sequence of intervals $\{J(A)\}_A$ complementary to $E_A$ such that $J(A) \to \ell$ in the sense of endpoint convergence. By construction, the intervals $J(A)$ are complementary to $E'_A$ also. Since the distance from $z$ to the endpoints of $\ell$ is positive, and the endpoints of $J(A)$ converge to the endpoints of $\ell$, it follows that the distance from $z$ to $E'_A$ is bounded from below as $A \to \infty$. Hence $z \not\in (E'_A)^\delta$ for some fixed $\delta > 0$ sufficiently small and all large $A$. By part $(ii)$ of \thref{L:FinalLemmaNecessity}, we obtain $z \not\in E'_\infty$, and so $z \in \T \setminus E'_\infty$. 
\end{proof}

From \eqref{E:RefinedInconsistencyIneq} and the non-negativity of $\nu$ we deduce that 
\begin{equation}
\label{E:limsupREAPrimeIntegrals}
\limsup_{A \to \infty} \int_{E'_A} R \, dm < \limsup_{A \to \infty} \, \nu(E_A) \leq \nu(\T).
\end{equation} We will show that this implies the inequality \begin{equation}
    \label{E:E'inftyRassociated}
    \int_{E'_\infty} R \, dm \leq \nu(\T)
\end{equation} from which we will easily deduce that $E_\infty \in \assoch{h}{w}$. If $|E'_\infty| = 0$, then the above inequality of course trivially holds, since the left-hand side is equal to $0$. By \thref{L:EE'equalityLemma} we have in that case $|E_\infty| = 0$ also, and hence $E_\infty \in \assoch{h}{w}$. In the case that $|E'_\infty| > 0$, by part $(i)$ of \thref{L:FinalLemmaNecessity}, for any small $\delta > 0$ and all large enough $A$, we have that $|E'_\infty| - \delta \leq |E'_A|$ and $E'_A \subset (E'_\infty)^\delta$. By fixing $\delta$, this implies for large $A$ the inequalities
\begin{align*}
    |E_\infty'| - \delta \leq |E'_A| &= |(E'_A \cap E'_\infty)| + |E'_A \setminus E'_\infty| \\
    &\leq |E'_\infty|  + |\big(E'_\infty)^\delta \setminus E'_\infty|.
\end{align*} 
Since $|\big(E'_\infty)^\delta \setminus E'_\infty| \to 0$ as $\delta \to 0$, the above inequalities imply that
\[ |E'_\infty| = \lim_{A \to \infty} |E'_A \cap E'_\infty| .\] Therefore 
\begin{equation}
    \label{E:EnFillingOutEinfty}
    E'_\infty = \big(E'_A \cap E'_\infty \big) \cup r'_A
\end{equation} where $|r'_A| \to 0$. For any $M > 0$ we consequently have \[\int_{r'_A} \min(R, M)\, d\m \to 0\] as $A \to \infty$, and since
\begin{align*}
    \int_{E'_\infty} \min(R, M)\, d\m &= \int_{E_A' \cap E'_\infty} \min(R, M)\, d\m + \int_{r'_A} \min(R, M)\, d\m
    \\  &\leq \int_{E'_A} R \, dm + \int_{r'_A} \min(R, M)\, d\m 
\end{align*}
we obtain from \eqref{E:limsupREAPrimeIntegrals} that
\[ \int_{E'_\infty} \min(R, M)\, d\m \leq \nu(\T).\] Letting $M \to \infty$ we arrive at \eqref{E:E'inftyRassociated}, which is equivalent to $E'_\infty \in \assoch{h}{w}$. It follows that $E_\infty \in \assoch{h}{w}$, since $\int_{E_\infty} R \, d\m = \int_{E'_\infty} R \, d\m$, owing to the fact that by \thref{L:EE'equalityLemma} the sets $E_\infty$ and $E'_\infty$ differ only by a set of Lebesgue measure zero. Since $E_A \subset E_\infty^\delta$ for all large $A$ and any $\delta > 0$, the inequalities in \eqref{E:RefinedInconsistencyIneq} and the non-negativity of $\nu$ show that 
\[ 0 < c_0 \leq \limsup_{A \to \infty} \nu(E_A) \leq \nu(E_\infty^\delta).\]
Finally, by continuity of finite measures, we obtain \[ \nu(E_\infty) = \lim_{\delta \to 0} \nu(E^\delta_\infty) \geq c_0 > 0.\]
We have thus shown that the inconsistency of \eqref{E:LinearProgram} implies that $\nu$ assigns positive mass to the set $E_\infty \in \assoch{h}{w}$. With this, our cyclicity proof is complete.

\subsection{Proof of the corollary}

Having completed the proof of \thref{T:CyclicityMainTheorem}, we give a proof of the corollary stated in the Introduction. We repeat the statement for convenience.

\begin{cor*}
    Let $f = BS_\nu U$ be the inner-outer factorization of a function $f \in \hil^\infty$, where $B$ is a Blaschke product, $S_\nu$ is a singular inner function, and $U$ is an outer function. Let $\nu = \nu_p + \nu_c$ be the decomposition of $\nu$ in \eqref{E:NuDecomp} associated to the weight $w$ and the gauge function $h$ in \eqref{E:CarlesonGaugeDef}. If $\mu$ has the form \eqref{E:MuDef} and $[f]$ denotes the smallest closed $M_z$-invariant subspace containing $f$ in $\Po^t(\mu)$, then \[[f] = [BS_{\nu_p}],\] and every function $h \in [f] \cap \hil^\infty$ satisfies $h/BS_{\nu_p} \in \hil^\infty$.
\end{cor*}

\begin{proof}
    It is clear that $[f] \subseteq [BS_{\nu_p}]$, since $f/BS_{\nu_p}$ is bounded, and so $f = (f/BS_{\nu_p})BS_{\nu_p} \in [BS_{\nu_p}]$. Conversely, if $g \in [BS_{\nu_p}]$, then there exists a sequence $\{p_n\}_n$ of polynomials for which $BS_{\nu_p} p_n \to g$ in $\Po^t(\mu)$. Since $S_{\nu_c}$ and $U$ are bounded functions, and both are cyclic in $\Po^t(\mu)$, the argument in the second remark following \thref{P:CyclicityProp} shows that $S_\nu U$ is cyclic in $\Po^t(\mu)$, and so there exists a sequence $\{q_k\}_k$ of polynomials for which $S_{\nu_c}Uq_k \to 1$ in $\Po^t(\mu)$. Then 
    \begin{align*}
        \|fp_nq_k - g\|_{\mu,t} &= \|BS_{\nu_p}p_n(S_{\nu_c}Uq_k - 1) + BS_{\nu_p}p_n - g\| \\
        &\leq \|BS_{\nu_p}p_n\|_\infty \|S_{\nu_c}Uq_k - 1\|_{\mu,t} + \|BS_{\nu_p}p_n - g\|_{\mu,t}.
    \end{align*} Fixing a large $n$ we make the second term arbitrarily small. Next, after fixing $n$, we may fix also a large $k$ to make the first term arbitrarily small. It follows that $g \in [f]$, and so we have the equality $[f] = [BS_{\nu_p}]$. The divisibility statement follows from \thref{P:PermanenceInnerAssocSet}. Indeed, recall that $\nu_p$ is supported on a countable union $\{E_n\}_n$ of sets in $\assoc{w}$, and we may assume that the sets $E_n$ increase with $n$. If $\nu_n$ is the restriction of $\nu_p$ to $E_n$, then $\nu_n \to \nu_p$ in the sense of weak*-convergence of measures. For any function $h \in [f] \cap \hil^\infty = [BS_{\nu_p}] \cap \hil^\infty \subset [S_{\nu_n}] \cap \hil^\infty$ we have by \thref{P:PermanenceInnerAssocSet} and the (generalized) maximum principle in $\hil^\infty$ that \[ \|h/S_{\nu_n}\|_\infty = \|h\|_\infty.\] By weak*-convergence and formula \eqref{E:SnuEq}, we have that $S_{\nu_n}(z) \to S_\nu(z)$ for $z \in \D$. Let $n \to \infty$ to obtain that 
    \[ \|h/S_{\nu_p}\|_\infty = \|h\|_\infty.\] Since certainly we have $h/B \in \hil^\infty$, we conclude finally that $h/BS_{\nu_p} \in \hil^\infty$. 
\end{proof}

\subsection{Nevanlinna class and cyclicity} As alluded to in the Introduction, the cyclicity result has a simple extension from bounded functions to quotients of bounded functions, i.e, to the Nevanlinna class:
\[ f = d/c, \quad d, c\in \hil^\infty, \, c(z) \neq 0, z \in \D.\] A function $f$ of this form has many other representation as a quotient of bounded functions, and a particularly useful one is
\begin{equation}
\label{E:NevRepr}
f = \frac{B S_{\nu_1} d_o}{S_{\nu_2} c_o},
\end{equation} where $d_o$ and $c_o$ are bounded outer functions, $B$ is a Blaschke product, and $\nu_1, \nu_2$ are singular inner functions for which the corresponding measures $\nu_1$ and $\nu_2$ are mutually singular. In other words, $S_{\nu_1}$ and $S_{\nu_2}$ have no common non-trivial singular inner factor.

\begin{cor} Let $f$ have the form \eqref{E:NevRepr}, with $B \equiv 1$, so that $f$ is non-vanishing in $\D$. If $f \in \Po^t(\mu)$ and $\mu$ has the form \eqref{E:MuDef}, then $f$ is cyclic in $\Po^t(\mu)$ if and only if $S_{\nu_1}$ is cyclic in $\Po^t(\mu)$. 
\end{cor}

\begin{proof}
If $S_{\nu_1}$ is not cyclic, then according to our main result there exists a set $E \in \assoc{w}$ for which $\nu_1(E) > 0$, and any bounded function in the invariant subspace $[S_{\nu_1|E}]$ has an inner factor divisible by $S_{\nu_1|E}$. If $f$ is cyclic, then for some sequence $\{p_n\}_n$ of polynomials we have $fp_n \to 1$ in the norm of $\Po^t(\mu)$. Since multiplication by $S_{\nu_2}c_o$ is a bounded operation on $\Po^t(\mu)$, we obtain the norm convergence \[ S_{\nu_1}d_o S_{\nu_2} c_o p_n \to S_{\nu_2} c_o, \] and hence $S_{\nu_2}c_o \in [S_{\nu_1|E}]$. Since $c_0$ is outer, and $\nu_1|E$ and $\nu_2$ are mutually singular, the function $S_{\nu_2}c_o$ has the inner factor $S_{\nu_2}$ which is not divisible by $S_{\nu_1|E}$, and so we reach a contradiction. Thus $S_{\nu_1}$ not being cyclic implies that $f$ is not cyclic either.

Conversely, assume that $S_{\nu_1}$ is cyclic. Then $S_{\nu_1}d_o$ is also cyclic, and clearly $ S_{\nu_1}d_o = S_{\nu_2}c_o f \in [f]$. Hence \[ [f] \supseteq [S_{\nu_1}d_o] = \Po^t(\mu),\] and the proof is complete.

\end{proof}

\appendix

\section{Premeasures}
\label{sec:appendixA}

The use of premeasures in the context of cyclicity problems goes back to Korenblum's work in \cite{korenblum1975extension}, \cite{korenblum1977beurling}, where they were used in Poisson-type representations of harmonic functions which satisfy certain growth bounds in the unit disk $\D$. In the article, we study only the cyclicity of bounded functions, and as a consequence, our use of premeasures is not a necessity, but rather a convenience which allows for certain simplifications in proofs. 

\subsection{Basic properties}

\begin{defn} \thlabel{D:PremDef}
A \textit{normalized premeasure} $\sigma$ is a set function mapping intervals in $\T$ (closed, open or half-open, and the interval may reduce to a single point) into real numbers, with the following properties:

\begin{enumerate}[(P1)]
    \item $\sigma(\T) = 0$ \textit{(normalization)},
    \item  $\sigma(I_1 \cup I_2) = \sigma(I_1) + \sigma(I_2)$ for disjoint intervals $I_1, I_2 \subset \T$ \textit{(additivity)},
    \item if $\{I_n\}_n$ is a nested sequence of intervals shrinking to the empty set, then $\sigma(I_n) \to 0$ \textit{(continuity)}.
\end{enumerate}

If the last two properties are satisfied but not the first, then we will simply say that $\sigma$ is a \textit{premeasure}. 
\end{defn}

Korenblum in \cite{korenblum1975extension} and \cite{korenblum1977beurling} postulates a number of properties of premeasures. We will spend this section to argue for validity of those properties in a slightly increased generality, which is necessary for the rest of the article.

To every normalized premeasure $\sigma$ we may associate its primitive function $\rho_\sigma: (0, 2\pi] \to \R$, defined by 
\begin{equation} \label{E:PremeasurPrimitiveDef} \rho_\sigma(t) = \sigma(I_t), \quad I_t = [1, e^{it}), \, t \in (0, 2\pi]. \end{equation} Here $I_t = [1, e^{it})$ is the half-open interval of $\T$ starting at $1$ (inclusive) and ending at $e^{it}$ (exclusive). The characterizing properties of the primitive functions $\rho = \rho_\sigma$ arising in this way from normalized premeasures are the following.

\begin{prop} \thlabel{P:PremeasurePrimitiveProp}
Let $\sigma$ be a normalized premeasure. Then the primitive function in \eqref{E:PremeasurPrimitiveDef} satisfies the following properties:
\begin{enumerate}[(I)]
    \item if $t' \to t \in (0, 2\pi]$ from the left, then $\rho(t') \to \rho(t)$ \textit{(continuity from the left)},
    \item if $t' \to t \in [0, 2\pi)$ from the right, then  $\rho(t')$ converges to a limit $\rho(t+)$ \textit{(limits from the right)},
    \item $\rho(2\pi) = 0$.
\end{enumerate}
Conversely, every function $\rho$ satisfying (I), (II) and (III) corresponds to a normalized premeasure $\sigma$, defined uniquely by the requirement that
\begin{equation} \label{E:ReconstrSigma} \sigma(I_t) := \rho(t), \quad I_t = [1, e^{it}), \, t \in (0, 2\pi). \end{equation}
\end{prop}

The lemma has a straightforward proof which we skip. Given $\rho$ and the equation \eqref{E:ReconstrSigma}, one extends $\sigma$ to unions of any type of intervals using the continuity properties (I) and (II) and the requirement (P2). If (III) does not hold for $\rho$, then the corresponding construction will produce a premeasure $\sigma$ which is not normalized. 

It will be convenient to set a notation for the size of the jump of $\rho$ at a point $t$:
\begin{equation} \label{E:JumpDef}
J_\rho(t) := |\rho(t) - \rho(t+)| = |\rho(t) - \lim_{\substack{ t' \to t \\ t' > t}} \rho(t')|, \quad t \in (0, 2\pi).
\end{equation}

Our definitions ensure that $J_\sigma$ is large only at a finite number of points.

\begin{lem} \thlabel{L:FiniteJumpLemma} For every $\delta > 0$, there exist only a finite number of $t \in (0, 2\pi)$ for which we have $J_\sigma(t) > \delta$.
\end{lem}

\begin{proof}
Were the lemma not true, then for some $\delta > 0$ a sequence $\{t_n\}_n$ of distinct numbers would exist which converges to some $t^* \in [0, 2\pi]$ and for which we have $J(t_n) > \delta$ for all $n$. After passing to a subsequence, we may assume that the elements of the sequence $\{t_n\}_n$ are all either strictly larger or strictly smaller than $t^*$, say the latter. By the definition of $J_\rho$, for each $n$ there exists a number $t'_n$ satisfying $t_n < t'_n < t^*$ for which we have $|\rho(t_n)-\rho(t'_n)| > \delta$. If we interweave the sequences $\{t_n\}_n$ and $\{t'_n\}_n$ into a sequence $\{s_n\}_n$ then $s_n \to t$ from the left, but $\rho(s_n)$ does not converge, breaking property (I) (were $\{t_n\}_n$ all strictly larger than $t^*$, an analogous argument would break property (II) instead).
\end{proof}

A smooth function $h$ may be integrated against a normalized premeasure by employing the following definition:

\begin{equation}
    \label{E:IntegratePremeasure}
    \int_{[0,2\pi)} h(e^{it}) d\sigma(e^{it}) := -\int_{[0,2\pi)} \Big(\frac{d}{dt}h(e^{it})\Big) \rho_\sigma(t)\,dt.
\end{equation} We note that, as a consequence of \thref{L:FiniteJumpLemma}, the function $\rho_\sigma$ has only a countable number of discontinuities, and is therefore Riemann integrable. It follows that the right-hand side in \eqref{E:IntegratePremeasure} is well-defined in ordinary sense. If $\sigma$ is a bona fide finite Borel measure satisfying $\sigma(\T) = 0$, then both sides of \eqref{E:IntegratePremeasure} make sense and are equal. We extend the definition \eqref{E:IntegratePremeasure} to not necessarily normalized premeasures in the natural way by decomposing a premeasure $d\sigma$ as a sum $(d\sigma - \sigma(\T)d\m) + \sigma(\T)d\m$ of a normalized premeasure and a constant multiple of the Lebesgue measure. The integral we then define as the sum of the integrals corresponding to $\sigma - \sigma(\T)d\m$ and $d\m$, the second defined in the classical sense.

\subsection{Sequential compactness}

If a family of premeasures obeys a certain upper estimate, then the family is sequentially compact in the following weak sense.
\begin{lem}[\textbf{Helly selection principle}]
\thlabel{L:Helly}
If $h: [0,1] \to \R_+$ is a continuous increasing function, $h(0) = 0$, and $\{\sigma_n\}_n$ is a sequence of normalized premeasures which obey the upper estimate
\[ \sigma_n(I) \leq h(|I|)\] for all intervals $I$ on $\T$, then there exists a subsequence $\{\sigma_{n_k}\}_k$, a real constant $c$, and a normalized premeasure $\sigma$ satisfying $\sigma(I) \leq h(|I|)$ for all intervals $I$, such that
\begin{equation}
\label{E:AlmostConvJumpError}
\limsup_{k \to \infty} |\rho_{\sigma_{n_k}}(t) - \rho_\sigma(t) - c| \leq J_\rho(t), \quad t \in (0, 2\pi).
\end{equation} 
Moreover, in the above situation, we have 
\[ \lim_{k \to \infty} \int_{[0,2\pi)} g(e^{it}) d\sigma_{n_k}(e^{it}) = \int_{[0,2\pi)} g(e^{it}) d\sigma(e^{it})\] for all smooth functions $g: \T \to \mathbb{C}$.
\end{lem}

The principle is stated in a similar form and without proof in \cite{korenblum1975extension} for the particular choice of gauge function $h$ in \eqref{E:CarlesonGaugeDef}. The normalization assumption is necessary as evidenced by the sequence $\{-n\cdot d\m\}_n$, which obviously satisfies the upper estimate, but for which the sequence $\{\rho_{\sigma_n}\}_n$ converges pointwise to $-\infty$ on $(0,2\pi)$ as $n \to \infty$. The appearance of the constant $c$ in \eqref{E:AlmostConvJumpError} may be justified by considering a sequence $\{\sigma_n\}_n$ for which we have $\lim_{n \to \infty} \rho_{\sigma_n}(t) = c \neq 0$ for all $t$ close to $2\pi$. Take, for instance, the functions
\begin{align*}
\rho_{\sigma_n}(t) = \begin{cases} t, & t \in [0, 1], \\ 1, & t \in (1, 2\pi - 1/n) \\ 1 - n(t-2\pi+1/n), & t \in [2\pi - 1/n, 2\pi]
\end{cases}
\end{align*} Then it follows readily from \eqref{E:ReconstrSigma} that we have the upper estimate $\sigma_n(I) \leq |I|$. However, note that if a subsequence of $\{\rho_{\sigma_n}\}_n$ converges pointwise to a function $\rho$ which satisfies (I), then $\rho$ must also satisfy $\rho(2\pi) = 1$, breaking (III).

\begin{proof}[Proof of \thref{L:Helly}]
By properties (P1) and (P2) of premeasures, and by our hypothesis, we obtain for any interval $I$ the inequality \[-\sigma_n(I) = \sigma_n(\T \setminus I) \leq h(|\T|).\] Since $\sigma_n(I) \leq h(|I|) \leq h(|\T|)$ also, we obtain the uniform bound $|\rho_{\sigma_n}(t)| \leq h(|\T|)$ for $t \in [0, 2\pi)$. Hence the sets $\{ \rho_{\sigma_n}(t)\}_n$ are bounded in $\R$, and so by the usual diagonal process, we may pass to a subsequence and ensure that $\rho_{\sigma_n}(r) \to c_r$, $n \to \infty$, for every rational $r \in (0, 2\pi]$. Here $c_r$ is a real number of modulus at most $h(|\T|)$. 
For $t \in (0, 2\pi]$, we define
\begin{equation} \label{E:RhoDef} \widetilde{\rho}(t) := \lim_{\substack{ r \to t \\ r < t}} \, c_r.
\end{equation} 
The limit does exist. Indeed, for rationals $r_1$, $r_2$ satisfying $r_1 < r_2 < t$ and the half-open interval $I_{r_1,r_2} = [e^{ir_1}, e^{ir_2})$ between them, we have 
\[c_{r_2} - c_{r_1} = \lim_{n \to \infty} \rho_{\sigma_n}(r_2) - \rho_{\sigma_n}(r_1) = \lim_{n \to \infty} \sigma_n(I_{r_1, r_2})\leq h(r_2-r_1) \]
If we let $r_2 \to t$ and $r_1 \to t$ from the left in such a way that \[ c_{r_2} \to \limsup_{\substack{ r \to t \\ r < t}} c_r, \quad c_{r_1} \to \liminf_{\substack{ r \to t \\ r < t}}  c_r\] then we obtain from the above inequality, and the hypothesis on $h$, that \[ \limsup_{\substack{ r \to t \\ r < t}} c_r - \liminf_{\substack{ r \to t \\ r < t}} c_r \leq 0.\]
Hence \eqref{E:RhoDef} is well-defined. 

We claim that $\widetilde{\rho}$ satisfies conditions (I) and (II) for being a primitive function of a premeasure which appear in \thref{P:PremeasurePrimitiveProp}. To verify (I), let $\{t_m\}_m$ be any sequence which converges to $t \in (0, 2\pi]$ from the left. Take rational $r_m$ such that $r_m < t_m$, $r_m \to t$, and $c_{r_m} > \widetilde{\rho}(t_m) - 1/m$. Then our definitions, and the already established existence of the limit in \eqref{E:RhoDef}, imply that \[ \widetilde{\rho}(t) = \lim_{m \to \infty} c_{r_m} \geq \limsup_{m \to \infty} \widetilde{\rho}(t_m). \] On the other hand, by \eqref{E:RhoDef} there exists a short open interval $I_\epsilon$ with right endpoint $t$ such that $c_r > \widetilde{\rho}(t) - \epsilon$ for every rational $r \in I_\epsilon$. If $t_m \in I_\epsilon$, then the definition \eqref{E:RhoDef} shows that $\widetilde{\rho}(t_m) \geq \widetilde{\rho}(t) - \epsilon$. Hence $\liminf_{m \to \infty} \widetilde{\rho}(t_m) \geq \widetilde{\rho}(t)$, and we have thus verified property (I).

We now show that (II) holds. Assume that the limit does not exist, and so that there exist two sequences $\{t_m\}_n$ and $\{t'_m\}_m$ which both converge to $t \in [0, 2\pi)$ from the right, but such that $\lim_{m \to \infty} \widetilde{\rho}(t_m) > \lim_{m \to \infty} \widetilde{\rho}(t'_m) + \delta$ for some $\delta > 0$. By our definition, this corresponds to existence of two sequences $\{r_m\}_m$, $\{r'_m\}_m$ of rational numbers converging to $t$ from the right, for which 
\[ \lim_{m \to \infty} c_{r_m} - c_{r'_m} > \delta. \]
We may assume that the sequences satisfy $r'_m < r_m$ for all $m$. Note that 
\[ \delta < c_{r_m} - c_{r'_m} = \lim_{n \to \infty} \rho_{\sigma_n}(r_m) - \rho_{\sigma_n}(r'_m) = \lim_{n \to \infty} \sigma_n(I_m), \quad I_m = [e^{ir'_m}, e^{ir_m}). \] Thus
\[ \delta < \lim_{n \to \infty} \sigma_n(I_m) \leq h(r'_m - r_m).\] Since $r'_m - r_m \to 0$ as $m \to \infty$, we obtain a contradiction to our hypothesis on $h$. It follows that (II) holds.

Let us now show that 
\begin{equation}
\label{E:JumpEqPrim}
\limsup_{n \to \infty} |p_{\sigma_n}(t) - \widetilde{\rho}(t)| \leq J_\rho(t), \quad t \in (0,2\pi).
\end{equation} Fix $\epsilon > 0$ and take rational $r$ for which we have $r < t$, $|t-r| < \epsilon$ and $|c_r - \widetilde{\rho}(t)| < \epsilon$. Then
\begin{align*}
\limsup_{n \to \infty} \, \rho_{\sigma_n}(t) - \widetilde{\rho}(t) &\leq \limsup_{n \to \infty}  \, \rho_{\sigma_n}(t) - c_r + \epsilon \\ &= \limsup_{n \to \infty}  \, \rho_{\sigma_n}(t) - \rho_{\sigma_n}(r) + \epsilon \\
&\leq h(\epsilon) + \epsilon.
\end{align*} Let $\epsilon \to 0$ to conclude that
\[ \limsup_{n \to \infty} \rho_{\sigma_n}(t) \leq \widetilde{\rho}(t), \quad t \in (0, 2\pi).\] Analogously, we may take $r_+ > t$, $|t - r_+| < \epsilon$, $|\rho(t+) - c_{r_+}| < \epsilon$, and estimate
\begin{align*}
\liminf_{n \to \infty} \,\widetilde{\rho}(t+) - \rho_{\sigma_n}(t) &\leq \liminf_{n \to \infty} \, c_{r_+} - \rho_{\sigma_n}(t) + \epsilon \\ &= \liminf_{n \to \infty} \, \rho_{\sigma_n}(r_+) - \rho_{\sigma_n}(t) + \epsilon \\
&\leq h(\epsilon) + \epsilon
\end{align*} which shows that \[ \widetilde{\rho}(t+) \leq \liminf_{n \to \infty} \rho_{\sigma_n}(t), \quad t \in (0,2\pi).\]
Combining the two inequalities, we arrive at 
\[ \widetilde{\rho}(t+) \leq \liminf_{n \to \infty} \rho_{\sigma_n}(t) \leq \limsup_{n \to \infty} \rho_{\sigma_n}(t) \leq \widetilde{\rho}(t), \quad t \in (0,2\pi)\] which implies \eqref{E:JumpEqPrim}.

Setting $c = \widetilde{\rho}(2\pi)$ and \[ \rho(t) = \widetilde{\rho}(t) - c, \quad t \in (0, 2\pi]\] we force condition (III), without breaking $(I)$ and $(II)$, and so we obtain a function $\rho = \rho_\sigma$ corresponding to a premeasure $\sigma$ as in \thref{P:PremeasurePrimitiveProp}. We deduce from \eqref{E:JumpEqPrim} that \eqref{E:AlmostConvJumpError} now holds, and moreover, at any point $t \in (0, 2\pi)$ at which $J_\rho(t) = 0$, we have that \[\lim_{n \to \infty} \rho_{\sigma_n}(t) = \rho_\sigma(t) + c.\] This implies the corresponding integral convergence. Indeed, let $g$ be a smooth function on $\T$. By the definition in \eqref{E:IntegratePremeasure}, uniform boundedness of $\rho_{\sigma_n}$ and the fact that $J_\rho(t)$ is non-zero only at a countable number of points (recall \thref{L:FiniteJumpLemma}), we obtain from the Dominated Convergence Theorem that
\begin{align*}
     \lim_{n \to \infty} \int_{[0,2\pi)} g(e^{it}) d\sigma_n(e^{it}) &= \lim_{n \to \infty} -\int_{[0,2\pi)} \Big(\frac{d}{dt}g(e^{it})\Big) \rho_{\sigma_n}(t)\,dt. \\
     &= -\int_{[0,2\pi)} \Big(\frac{d}{dt}g(e^{it})\Big) \rho_{\sigma}(t)\,dt - c\int_{[0,2\pi)}\frac{d}{dt}g(e^{it})\,dt \\
     &=-\int_{[0,2\pi)} \Big(\frac{d}{dt}g(e^{it})\Big) \rho_{\sigma}(t)\,dt \\
     &= \int_{[0,2\pi)} g(e^{it}) d\sigma(e^{it}).
\end{align*} We used that $g(e^{it})$ is periodic between the second and third lines, which makes the second integral on the second line vanish.

Finally, we verify that $\sigma(I) \leq h(|I|)$ for all intervals. By property (I) and (II) of premeasures, and continuity of $h$, it will suffice to show the bound for half-open intervals $I = [e^{it}, e^{it'})$ with endpoints $e^{it}$, $e^{it'}$ on which $\sigma$ carries no mass. Then the jumps $J_\rho(t)$, $J_\rho(t')$ are equal to $0$, and so from \eqref{E:AlmostConvJumpError} we obtain
\[ \sigma(I) = \rho(t) - \rho(t') = \lim_{n \to \infty} \rho_{\sigma_n}(t) - \rho_{\sigma_n}(t') = \lim_{n \to \infty} \sigma_n(I) \leq h(|I|).\]
\end{proof}

\subsection{Poisson integral estimates}

The Poisson integral bound \eqref{E:PoissonEstimate0} which we use in Section \ref{S:ThomsonDecompProofSec} holds also for normalized premeasures. Namely, if we have the upper estimate $\sigma(I) \leq h(|I|)$, then setting $g(e^{it}) = P_z(e^{it})$ to be the Poisson kernel,
\begin{equation}
    \label{E:PoissonKernelDef}
    P_z(e^{it}) := \frac{1-|z|^2}{|e^{it} - z|^2} =  \Re \Big(\frac{e^{it} +z}{e^{it}-z} \Big),
\end{equation} the same proof as in the case of $\sigma$ being an ordinary Borel measure (see \cite[p. 297]{havinbook}) gives us a constant $C = C(h) > 0$ for which it holds that
\begin{equation}
    \label{E:PoissonEstimate}
    \int_{[0,2\pi)} P_z(e^{it}) d\sigma(e^{it}) \leq \frac{C h(1-|z|)}{1-|z|}.
\end{equation} 
See also \cite[Exercise 4.9.1]{hedenmalmbergmanspaces}. A similar estimate holds of course also for non-normalized premeasures, but then the constant $C$ depends on the quantity $\sigma(\T)$ also: $C = C(h)\exp(\sigma(\T))$.

\bibliographystyle{plain}
\bibliography{mybib}

\end{document}